\documentclass[letterpaper,11pt,reqno]{amsart}

\usepackage[T1]{fontenc}
\usepackage[utf8x]{inputenc}
\usepackage{hyperref}
\usepackage{xcolor}
\hypersetup{
    colorlinks,
    linkcolor={red!50!black},
    citecolor={blue!50!black},
    urlcolor={blue!80!black}
}

\usepackage{microtype}
\usepackage{MnSymbol}
\usepackage{textcomp}
\usepackage{stmaryrd}

\usepackage[mathcal]{euscript}
\let\mathcaltmp\mathcal
\usepackage{mathrsfs}
\let\mathcal\mathscr
\let\mathscr\mathcaltmp

\makeatletter
\newcommand{\eqnum}{\refstepcounter{equation}\textup{\tagform@{\theequation}}}
\makeatother

\makeatletter \@addtoreset{equation}{section} \makeatother
\renewcommand{\theequation}{\thesection.\arabic{equation}}

\newtheorem{thm}[equation]{Theorem}
\newtheorem{thmX}{Theorem}

\newtheorem{corX}[thmX]{Corollary}

\newtheorem*{thm*}{Theorem}
\newtheorem{lem}[equation]{Lemma}
\newtheorem{cor}[equation]{Corollary}
\newtheorem{prop}[equation]{Proposition}
\newtheorem*{defthm*}{Definition/Theorem}

\theoremstyle{definition}
\newtheorem{defn}[equation]{Definition}
\newtheorem{rem}[equation]{Remark}
\newtheorem{exam}[equation]{Example}
\newtheorem{constr}[equation]{Construction}
\newtheorem{varnt}[equation]{Variant}

\newtheorem*{exam*}{Example}

\newcommand\arXiv[1]{\href{http://arxiv.org/abs/#1}{arXiv:#1}}

\usepackage{xparse}

\usepackage{xspace}

\newcommand{\nc}{\newcommand}
\nc{\renc}{\renewcommand}
\nc{\ssec}{\subsection}
\nc{\sssec}{\subsubsection}
\nc{\on}{\operatorname}
\nc{\term}[1]{#1\xspace}

\usepackage{tikz}
\usetikzlibrary{matrix}
\usepackage{tikz-cd}

\tikzset{
  commutative diagrams/.cd,
  arrow style=tikz,
  diagrams={>=latex}}
\makeatletter
\tikzset{
  column sep/.code=\def\pgfmatrixcolumnsep{\pgf@matrix@xscale*(#1)},
  row sep/.code   =\def\pgfmatrixrowsep{\pgf@matrix@yscale*(#1)},
  matrix xscale/.code=%
    \pgfmathsetmacro\pgf@matrix@xscale{\pgf@matrix@xscale*(#1)},
  matrix yscale/.code=%
    \pgfmathsetmacro\pgf@matrix@yscale{\pgf@matrix@yscale*(#1)},
  matrix scale/.style={/tikz/matrix xscale={#1},/tikz/matrix yscale={#1}}}
\def\pgf@matrix@xscale{1}
\def\pgf@matrix@yscale{1}
\makeatother

\usepackage[inline]{enumitem}
\setlist[enumerate,1]{label={(\alph*)},itemsep=\parskip}

\newlist{enumcompress}{enumerate}{1}
\setlist[enumcompress,1]{
  label={(\alph*)},
  itemsep=0.3\parskip,
  leftmargin=*,
  align=left,
  topsep=0em
}

\newlist{thmlist}{enumerate}{1}
\setlist[thmlist,1]{
  label={\em(\roman*)}, ref={(\roman*)},
  itemsep=0.5em,
  leftmargin=*,
  align=right,widest=vi)}

\newlist{thmlistbis}{enumerate}{1}
\setlist[thmlistbis,1]{
  label={\em(\roman*~\textit{bis})},
  ref={(\roman*}~\textit{bis}\upshape{)},
  itemsep=0.5em,
  leftmargin=0pt, align=right, widest=vi)}

\newlist{defnlist}{enumerate}{2}
\setlist[defnlist,1]{
  label={(\roman*)}, ref={(\roman*)},
  itemsep=0.5em,
  leftmargin=*,
  align=right, widest=vi)}

\setlist[defnlist,2]{
  label={(\alph*)}, ref={(\alph*)},
  itemsep=0.75em,
  labelsep=0em,labelindent=0em,leftmargin=*,align=left,widest=vi),
  topsep=0.75em}

\newlist{inlinelist}{enumerate*}{1}
\setlist[inlinelist,1]{label={(\alph*)}}

\newlist{inlinedefnlist}{enumerate*}{1}
\definecolor{green}{HTML}{38550C}
\setlist[inlinedefnlist,1]{label={\color{green}(\roman*)}}

\nc{\cA}{\ensuremath{\mathcal{A}}\xspace}
\nc{\cB}{\ensuremath{\mathcal{B}}\xspace}
\nc{\cC}{\ensuremath{\mathcal{C}}\xspace}
\nc{\cD}{\ensuremath{\mathcal{D}}\xspace}
\nc{\cE}{\ensuremath{\mathcal{E}}\xspace}
\nc{\cF}{\ensuremath{\mathcal{F}}\xspace}
\nc{\cG}{\ensuremath{\mathcal{G}}\xspace}
\nc{\cH}{\ensuremath{\mathcal{H}}\xspace}
\nc{\cI}{\ensuremath{\mathcal{I}}\xspace}
\nc{\cJ}{\ensuremath{\mathcal{J}}\xspace}
\nc{\cK}{\ensuremath{\mathcal{K}}\xspace}
\nc{\cL}{\ensuremath{\mathcal{L}}\xspace}
\nc{\cM}{\ensuremath{\mathcal{M}}\xspace}
\nc{\cN}{\ensuremath{\mathcal{N}}\xspace}
\nc{\cO}{\ensuremath{\mathcal{O}}\xspace}
\nc{\cP}{\ensuremath{\mathcal{P}}\xspace}
\nc{\cQ}{\ensuremath{\mathcal{Q}}\xspace}
\nc{\cR}{\ensuremath{\mathcal{R}}\xspace}
\nc{\cS}{\ensuremath{\mathcal{S}}\xspace}
\nc{\cT}{\ensuremath{\mathcal{T}}\xspace}
\nc{\cU}{\ensuremath{\mathcal{U}}\xspace}
\nc{\cV}{\ensuremath{\mathcal{V}}\xspace}
\nc{\cW}{\ensuremath{\mathcal{W}}\xspace}
\nc{\cX}{\ensuremath{\mathcal{X}}\xspace}
\nc{\cY}{\ensuremath{\mathcal{Y}}\xspace}
\nc{\cZ}{\ensuremath{\mathcal{Z}}\xspace}

\nc{\sA}{\ensuremath{\mathscr{A}}\xspace}
\nc{\sB}{\ensuremath{\mathscr{B}}\xspace}
\nc{\sC}{\ensuremath{\mathscr{C}}\xspace}
\nc{\sD}{\ensuremath{\mathscr{D}}\xspace}
\nc{\sE}{\ensuremath{\mathscr{E}}\xspace}
\nc{\sF}{\ensuremath{\mathscr{F}}\xspace}
\nc{\sG}{\ensuremath{\mathscr{G}}\xspace}
\nc{\sH}{\ensuremath{\mathscr{H}}\xspace}
\nc{\sI}{\ensuremath{\mathscr{I}}\xspace}
\nc{\sJ}{\ensuremath{\mathscr{J}}\xspace}
\nc{\sK}{\ensuremath{\mathscr{K}}\xspace}
\nc{\sL}{\ensuremath{\mathscr{L}}\xspace}
\nc{\sM}{\ensuremath{\mathscr{M}}\xspace}
\nc{\sN}{\ensuremath{\mathscr{N}}\xspace}
\nc{\sO}{\ensuremath{\mathscr{O}}\xspace}
\nc{\sP}{\ensuremath{\mathscr{P}}\xspace}
\nc{\sQ}{\ensuremath{\mathscr{Q}}\xspace}
\nc{\sR}{\ensuremath{\mathscr{R}}\xspace}
\nc{\sS}{\ensuremath{\mathscr{S}}\xspace}
\nc{\sT}{\ensuremath{\mathscr{T}}\xspace}
\nc{\sU}{\ensuremath{\mathscr{U}}\xspace}
\nc{\sV}{\ensuremath{\mathscr{V}}\xspace}
\nc{\sW}{\ensuremath{\mathscr{W}}\xspace}
\nc{\sX}{\ensuremath{\mathscr{X}}\xspace}
\nc{\sY}{\ensuremath{\mathscr{Y}}\xspace}
\nc{\sZ}{\ensuremath{\mathscr{Z}}\xspace}

\nc{\bA}{\ensuremath{\mathbf{A}}\xspace}
\nc{\bB}{\ensuremath{\mathbf{B}}\xspace}
\nc{\bC}{\ensuremath{\mathbf{C}}\xspace}
\nc{\bD}{\ensuremath{\mathbf{D}}\xspace}
\nc{\bE}{\ensuremath{\mathbf{E}}\xspace}
\nc{\bF}{\ensuremath{\mathbf{F}}\xspace}
\nc{\bG}{\ensuremath{\mathbf{G}}\xspace}
\nc{\bH}{\ensuremath{\mathbf{H}}\xspace}
\nc{\bI}{\ensuremath{\mathbf{I}}\xspace}
\nc{\bJ}{\ensuremath{\mathbf{J}}\xspace}
\nc{\bK}{\ensuremath{\mathbf{K}}\xspace}
\nc{\bL}{\ensuremath{\mathbf{L}}\xspace}
\nc{\bM}{\ensuremath{\mathbf{M}}\xspace}
\nc{\bN}{\ensuremath{\mathbf{N}}\xspace}
\nc{\bO}{\ensuremath{\mathbf{O}}\xspace}
\nc{\bP}{\ensuremath{\mathbf{P}}\xspace}
\nc{\bQ}{\ensuremath{\mathbf{Q}}\xspace}
\nc{\bR}{\ensuremath{\mathbf{R}}\xspace}
\nc{\bS}{\ensuremath{\mathbf{S}}\xspace}
\nc{\bT}{\ensuremath{\mathbf{T}}\xspace}
\nc{\bU}{\ensuremath{\mathbf{U}}\xspace}
\nc{\bV}{\ensuremath{\mathbf{V}}\xspace}
\nc{\bW}{\ensuremath{\mathbf{W}}\xspace}
\nc{\bX}{\ensuremath{\mathbf{X}}\xspace}
\nc{\bY}{\ensuremath{\mathbf{Y}}\xspace}
\nc{\bZ}{\ensuremath{\mathbf{Z}}\xspace}

\nc{\bbA}{\ensuremath{\mathbb{A}}\xspace}
\nc{\bbB}{\ensuremath{\mathbb{B}}\xspace}
\nc{\bbC}{\ensuremath{\mathbb{C}}\xspace}
\nc{\bbD}{\ensuremath{\mathbb{D}}\xspace}
\nc{\bbE}{\ensuremath{\mathbb{E}}\xspace}
\nc{\bbF}{\ensuremath{\mathbb{F}}\xspace}
\nc{\bbG}{\ensuremath{\mathbb{G}}\xspace}
\nc{\bbH}{\ensuremath{\mathbb{H}}\xspace}
\nc{\bbI}{\ensuremath{\mathbb{I}}\xspace}
\nc{\bbJ}{\ensuremath{\mathbb{J}}\xspace}
\nc{\bbK}{\ensuremath{\mathbb{K}}\xspace}
\nc{\bbL}{\ensuremath{\mathbb{L}}\xspace}
\nc{\bbM}{\ensuremath{\mathbb{M}}\xspace}
\nc{\bbN}{\ensuremath{\mathbb{N}}\xspace}
\nc{\bbO}{\ensuremath{\mathbb{O}}\xspace}
\nc{\bbP}{\ensuremath{\mathbb{P}}\xspace}
\nc{\bbQ}{\ensuremath{\mathbb{Q}}\xspace}
\nc{\bbR}{\ensuremath{\mathbb{R}}\xspace}
\nc{\bbS}{\ensuremath{\mathbb{S}}\xspace}
\nc{\bbT}{\ensuremath{\mathbb{T}}\xspace}
\nc{\bbU}{\ensuremath{\mathbb{U}}\xspace}
\nc{\bbV}{\ensuremath{\mathbb{V}}\xspace}
\nc{\bbW}{\ensuremath{\mathbb{W}}\xspace}
\nc{\bbX}{\ensuremath{\mathbb{X}}\xspace}
\nc{\bbY}{\ensuremath{\mathbb{Y}}\xspace}
\nc{\bbZ}{\ensuremath{\mathbb{Z}}\xspace}

\nc{\mrm}[1]{\ensuremath{\mathrm{#1}}\xspace}
\nc{\mfr}[1]{\ensuremath{\mathfrak{#1}}\xspace}
\nc{\mit}[1]{\ensuremath{\mathit{#1}}\xspace}
\nc{\mbf}[1]{\ensuremath{\mathbf{#1}}\xspace}
\nc{\mcal}[1]{\ensuremath{\mathcal{#1}}\xspace}
\nc{\msc}[1]{\ensuremath{\mathscr{#1}}\xspace}

\nc{\sub}{\subseteq}
\nc{\too}{\longrightarrow}
\nc{\hook}{\hookrightarrow}
\nc{\hooklongrightarrow}{\lhook\joinrel\longrightarrow}
\nc{\hooklong}{\hooklongrightarrow}
\nc{\hooklongleftarrow}{\longleftarrow\joinrel\rhook}
\nc{\twoheadlongrightarrow}{\relbar\joinrel\twoheadrightarrow}
\nc{\longrightleftarrows}{\ \raisebox{0.3ex}{\(\mathrel{\substack{\xrightarrow{\rule{1em}{0em}} \\[-1ex] \xleftarrow{\rule{1em}{0em}}}}\)}\ }

\renc{\ge}{\geqslant}
\renc{\geq}{\geqslant}
\renc{\le}{\leqslant}
\renc{\leq}{\leqslant}

\nc{\id}{\mathrm{id}}

\DeclareMathOperator{\rk}{\mathrm{rk}}
\DeclareMathOperator{\Hom}{\on{Hom}}
\nc{\uHom}{\underline{\smash{\Hom}}}
\DeclareMathOperator{\Maps}{\on{Maps}}
\DeclareMathOperator{\Aut}{\on{Aut}}
\DeclareMathOperator{\End}{\on{End}}

\nc{\uEnd}{\underline{\smash{\End}}}

\nc{\colim}{\varinjlim}
\renc{\lim}{\varprojlim}
\nc{\Cofib}{\on{Cofib}}
\nc{\Fib}{\on{Fib}}
\nc{\initial}{\varnothing}
\nc{\op}{\mathrm{op}}

\DeclareMathOperator*{\fibprod}{\times}

\renc{\setminus}{\smallsetminus}

\usepackage{mathtools}

\newcommand{\thmref}[1]{Theorem~\ref{#1}}

\newcommand{\secref}[1]{Sect.~\ref{#1}}
\newcommand{\ssecref}[1]{Subsect. ~\ref{#1}}
\newcommand{\sssecref}[1]{(\ref{#1})}
\newcommand{\lemref}[1]{Lemma~\ref{#1}}
\newcommand{\propref}[1]{Proposition~\ref{#1}}
\newcommand{\corref}[1]{Corollary~\ref{#1}}

\newcommand{\remref}[1]{Remark~\ref{#1}}
\newcommand{\defnref}[1]{Definition~\ref{#1}}

\renewcommand{\eqref}[1]{(\ref{#1})}
\newcommand{\constrref}[1]{Construction~\ref{#1}}
\newcommand{\varntref}[1]{Variant~\ref{#1}}

\newcommand{\examref}[1]{Example~\ref{#1}}

\newcommand{\itemref}[1]{\ref{#1}}

\nc{\A}{\bA}
\renc{\P}{\bP}
\nc{\Spec}{\on{Spec}}
\nc{\D}{\on{\mbf{D}}}
\nc{\Dqc}{\on{\mbf{D}}_{\mrm{qc}}}
\nc{\bDelta}{\mathbf{\Delta}}
\nc{\Cech}{\textnormal{\v{C}}}
\nc{\Dperf}{\on{\mbf{D}}_{\mrm{perf}}}
\nc{\Coh}{\on{Coh}}
\nc{\Qcoh}{\on{Qcoh}}
\nc{\Dcoh}{\on{\mbf{D}}_{\mrm{coh}}}
\nc{\uCoh}{\underline{\smash{\Coh}}}
\nc{\cl}{{\mrm{cl}}}
\nc{\Bl}{\on{Bl}}
\nc{\vir}{\mrm{vir}}
\nc{\CH}{\on{CH}}
\nc{\Nis}{\mrm{Nis}}
\nc{\et}{\mrm{\acute{e}t}}
\nc{\oH}{\on{H}}
\nc{\BM}{\mrm{BM}}
\nc{\Z}{\bZ}
\nc{\Q}{\bQ}
\nc{\K}{{\on{K}}}
\nc{\KGL}{\mrm{KGL}}
\nc{\MGL}{\mrm{MGL}}
\nc{\HZ}{\mrm{H}\bZ}
\nc{\KB}{\K^{\mrm{B}}}
\nc{\G}{{\on{G}}}
\nc{\KH}{\mrm{KH}}
\nc{\Ket}{\K^{\et}}
\nc{\KHet}{\KH^{\et}}
\nc{\Get}{\G^{\et}}
\nc{\rightrightrightarrows}
  {\mathrel{\substack{\textstyle\rightarrow\\[-0.6ex]
   \textstyle\rightarrow \\[-0.6ex]
   \textstyle\rightarrow}}}
\nc{\rightrightrightrightarrows}
  {\mathrel{\substack{\textstyle\rightarrow\\[-0.6ex]
   \textstyle\rightarrow \\[-0.6ex]
   \textstyle\rightarrow \\[-0.6ex]
   \textstyle\rightarrow}}}
\nc{\Einfty}{{\sE_\infty}}
\renc{\sp}{\mrm{sp}}
\nc{\cosp}{\mrm{cosp}}
\nc{\Td}{\on{Td}}
\nc{\ch}{\on{ch}}
\nc{\RGamma}{R\Gamma}
\nc{\red}{\mrm{red}}
\nc{\der}{{\mrm{der}}}
\nc{\Mod}{{\mrm{Mod}}}
\nc{\Gr}{{\on{Gr}}}
\nc{\Ind}{\on{Ind}}
\nc{\form}{\widehat}
\nc{\R}{\bR}
\renc{\L}{\bL}
\nc{\otimesL}{\mathchoice{\overset{\bL}{\otimes}}{\otimes^\bL}{\otimes^\bL}{\otimes^\bL}}
\nc{\fibprodR}{\fibprod^R}
\nc{\uRHom}{\bR\uHom}
\nc{\GL}{\mrm{GL}}
\nc{\SL}{\mrm{SL}}
\nc{\SW}{\on{SW}}
\nc{\Vect}{\on{Vect}}
\nc{\Fun}{\on{Fun}}
\nc{\vb}[1]{\langle #1\rangle}
\nc{\loc}{\mrm{loc}}
\nc{\fix}{\mrm{fix}}
\nc{\mov}{\mrm{mov}}
\nc{\cms}{\mrm{cms}}
\nc{\V}{\bV}
\nc{\Gm}{{\bG_m}}
\nc{\pt}{\mrm{pt}}
\nc{\mot}{\mrm{mot}}
\nc{\an}{\mrm{an}}
\nc{\St}{\mrm{St}}
\nc{\pr}{\mrm{pr}}
\nc{\C}{\on{C}}
\nc{\Chom}{\mrm{C}_\bullet}
\nc{\Ccoh}{\mrm{C}^\bullet}
\nc{\CBM}{\mrm{C}^{\BM}_\bullet}
\nc{\uAut}{\underline{\Aut}}
\nc{\per}{\vb{\ast}}
\nc{\Lis}{\mrm{Lis}}
\nc{\aff}{{\mrm{aff}}}
\nc{\dR}{{\mrm{dR}}}
\nc{\Pic}{{\on{Pic}}}
\nc{\uGrp}{\underline{\smash{\mrm{Grp}}}}
\nc{\uPerf}{\underline{\smash{\mrm{Perf}}}}
\nc{\uPic}{\underline{\smash{\Pic}}}
\nc{\uMap}{\underline{\smash{\mrm{Map}}}}
\nc{\uDiv}{\underline{\smash{\mrm{Div}}}}
\nc{\uPair}{\underline{\smash{\mrm{Pair}}}}
\nc{\dash}{\textnormal{-}}
\nc{\nilp}{\mrm{nilp}}
\nc{\Rep}{\on{R}}
\nc{\dashmod}{\dash\mrm{mod}}
\nc{\bLambda}{\mbf{\Lambda}}
\nc{\un}{\mbf{1}}
\nc{\Spt}{\mrm{Spt}}
\nc{\Stk}{\mrm{Stk}}
\nc{\tr}{\mrm{tr}}
\nc{\eul}{\mrm{eul}}
\nc{\gys}{\mrm{gys}}
\nc{\inv}{[\Sigma^{-1}]}
\nc{\Corr}{\on{Corr}}
\nc{\Catoo}{\mrm{Cat}_{\infty}}
\nc{\CatooL}{\mrm{Cat}_{\infty}^{\mrm{L}}}
\nc{\SH}{\on{SH}}
\nc{\MGLmod}{\on{D_\MGL}}
\nc{\KGLmod}{\on{D_\KGL}}
\nc{\lisse}{{\triangleleft}}
\nc{\gen}{{\mrm{gen}}}

\nc{\scr}{\term{derived commutative ring}}
\nc{\scrs}{\term{derived commutative rings}}

\nc{\inftyCat}{\term{$\infty$-category}}
\nc{\inftyCats}{\term{$\infty$-categories}}

\nc{\inftyGrpd}{\term{$\infty$-groupoid}}
\nc{\inftyGrpds}{\term{$\infty$-groupoids}}

\nc{\dA}{\term{derived Artin}}

\title{Categorical concentration\vspace{-2mm}}
\author[A.\,A. Khan]{Adeel A. Khan}
\author[C. Ravi]{Charanya Ravi}
\date{2026-11-19}

\makeatletter
\def\l@subsection{\@tocline{2}{0pt}{4pc}{6pc}{}}
\def\l@subsubsection{\@tocline{3}{0pt}{8pc}{8pc}{}}
\makeatother

\begin{document}

\begin{abstract}
  Given a torus action on a compact space $X$, a fundamental result of Borel and Atiyah--Segal asserts that the equivariant cohomology of $X$ is concentrated in the fixed locus $X^T$, up to inverting enough Chern classes.
  For algebraic varieties and stacks over an arbitrary field, we prove an analogous concentration statement for equivariant derived categories of sheaves.
  On the localized categories, restriction to the fixed locus coincides with Braden's hyperbolic localization, and as an application we obtain torus localization for the Donaldson--Thomas perverse sheaves of $(-1)$-shifted symplectic spaces.
  Concentration holds more generally for actions of affine algebraic groups, and for abstract six functor formalisms (e.g. motivic sheaves).
  For $\mu_\ell$-actions, it recovers Treumann's sheaf-theoretic Smith theory.
  We also prove a categorification of the virtual localization formula for torus actions.
  \vspace{-5mm}
\end{abstract}


\maketitle

\renewcommand\contentsname{\vspace{-1cm}}
\tableofcontents

\setlength{\parindent}{0em}
\parskip 0.75em

\thispagestyle{empty}

\section*{Introduction}

\ssec{Concentration}

  Let $k$ be a field.
  For an algebraic space $X$ of finite type over $k$, let $\D(X)$ denote one of the following:
  \begin{enumcompress}[label={(\roman*)}]
    \item the derived \inftyCat of sheaves of abelian groups on the topological space $X(\bC)$ (if $k=\bC$),
    \item the derived \inftyCat of sheaves of $\Z/n\Z$-modules on the small étale site $X_\et$ (if $n$ is invertible in $k$),
    \item the derived \inftyCat of $\ell$-adic sheaves on the small étale site $X_\et$ (if $\ell$ is invertible in $k$),
    \item the \inftyCat $\on{DM}(X)$ of relative Voevodsky motives over $X$, or other flavours of motivic sheaves (see \ssecref{ssec:weaves}).
  \end{enumcompress}
  Each of these admits a canonical extension to algebraic stacks (see \cite{weavelisse}).
  For an affine algebraic group $G$ acting on $X$, we write $\D^G(X) := \D([X/G])$, and $\oH^*_G(X; \cF)$ for the $G$-equivariant cohomology of $X$ with coefficients in $\cF \in \D^G(X)$\footnote{%
    In cases (ii)--(iv), $\oH^*_G(X; \cF)$ denotes the direct sum of the groups $\oH^*([X/G]; \cF(j))$ over all Tate twists $j\in\Z$, and similarly for $\on{Ext}^*$.
  }.
  Given a set $\Sigma$ of rank one $G$-representations, the localization $\D^G(X)\inv$ is defined by formally inverting the first Chern class $c_1(\rho) : \cF \to \cF(1)[2]$ for every $\rho\in\Sigma$ and every $\cF \in \D^G(X)$ (see \ssecref{ssec:weaves/sigmalocal}).
  Then we have (see \thmref{thm:equiconcat}):

  \begin{thmX}[Equivariant concentration]\label{thm:intro/equi}
    Let $i : Z \hook X$ be a $G$-invariant closed subspace containing every geometric point $x$ whose stabilizer $\St_x$ acts nontrivially through every $\rho\in\Sigma$.
    Then the functor
    \[ i_* : \D^G(Z)\inv \to \D^G(X)\inv \]
    is an equivalence of \inftyCats.
  \end{thmX}

  For $G = T$ a split torus and $\Sigma$ the set of nontrivial characters of $T$, we may take $Z$ to be the fixed locus $X^T$.

  \thmref{thm:intro/equi} is the case $\sX = [X/G]$, with $\Sigma$ pulled back from $BG$, of the following statement, in which $\Sigma$ is an arbitrary set of line bundles on an algebraic stack $\sX$ and the stabilizers $\St_x$ are replaced by the automorphism groups of the points of $\sX$ (see \thmref{thm:concat}):

  \begin{thmX}[Concentration on stacks]\label{thm:intro/concat}
    Let $\sX$ be an algebraic stack of finite type over $k$ with affine stabilizers, $\Sigma$ a set of line bundles on $\sX$, and $\sZ \sub \sX$ a closed substack containing every geometric point $x$ of $\sX$ for which the restrictions $\cL|_{B\uAut_\sX(x)}$ are nontrivial for all $\cL\in\Sigma$.
    Then the inclusion $i : \sZ \hook \sX$ induces an equivalence of \inftyCats
    \[ i_* : \D(\sZ)\inv \to \D(\sX)\inv. \]
  \end{thmX}

  On Ext groups, this gives (see \corref{cor:extloc}):

  \begin{corX}\label{cor:intro/ext}
    In the situation of \thmref{thm:intro/concat}, there are canonical isomorphisms for all $\cF, \cG \in \D(\sX)$,
    \begin{equation*}
      \on{Ext}^*(\cF,\cG)\inv \simeq \on{Ext}^*(i^*\cF,i^*\cG)\inv,
    \end{equation*}
    and similarly with $i^!$ in place of $i^*$.
  \end{corX}

  In particular, in the situation of \thmref{thm:intro/equi}, for every $\cF \in \D^G(X)$ we obtain cohomological equivariant concentration isomorphisms
  \begin{equation}\label{eq:intro/coh con}
    \oH^*_G(X; \cF)\inv \simeq \oH^*_G(Z; i^*\cF)\inv,
  \end{equation}
  and similarly for $i^!$ and for cohomology with compact supports.

  Special cases of these isomorphisms were known in the following settings, under further assumptions on $X$ and $\cF$:
  \begin{enumcompress}
    \item \emph{Constant coefficients.}
    For complex varieties, the isomorphism \eqref{eq:intro/coh con} is a special case of the concentration theorem of Borel and Atiyah--Segal for compact Lie group actions (see \cite[\S 3.2, Thm.~III.1]{Hsiang}).
    \item \emph{Sheaf coefficients, torus actions.}
    For split torus actions on complex varieties, \eqref{eq:intro/coh con} was proven for sheaves of $\bC$-vector spaces in \cite[Thm.~6.2]{GoreskyKottwitzMacPherson}, and for sheaves of $\Lambda$-modules with $\Lambda$ a factorial ring in \cite[Thm.~3.2]{FiebigWilliamson}.
    The $!$-version, and compactly supported $*$-version, appear in \cite[Thm.~B.2]{EvensMirkovic}.
    The more general isomorphisms for Ext groups, as in \corref{cor:intro/ext}, are stated in \cite[\S\S 3.1--3.2]{BravermanFinkelbergNakajima}.
    \item \emph{Borel--Moore homology and Chow groups.}
    For split torus actions on algebraic $k$-schemes, Edidin and Graham proved concentration for the equivariant Chow groups with rational coefficients in \cite[Thm.~1]{EdidinGrahamLocalization}.
    This was extended to Artin stacks in \cite[Thm.~B]{constack}.
    More generally, the special case of \corref{cor:intro/ext} for $\cF = \un$ and $\cG$ the dualizing complex recovers concentration for the Borel--Moore homology as proven in \cite[Thm.~A]{constack}.
  \end{enumcompress}
  After this work first appeared, Haution \cite[Thm.~2]{Haution} proved an equivalence of the form \thmref{thm:intro/equi} for the genuine equivariant motivic stable homotopy category $\SH^G(X)$, for $X$ $G$-quasi-projective and $G$ of multiplicative type\footnote{%
    This assumption can be dropped by using the theory in \cite{Sixstack} (since $[X/G]$ is nicely scalloped by Theorem~2.14 there).
  }, after inverting the Euler classes of all $G$-representations with no trivial summand.

\ssec{The functor \texorpdfstring{$i^{!*}$}{i!*}}

  In the situation of \thmref{thm:intro/concat}, the $!$- and $*$-restriction functors $i^!$ and $i^*$ collapse into the single \emph{$\Sigma$-localized restriction} functor
  \begin{equation}\label{eq:intro/shriekstar}
    i^{!*} := L_\Sigma \circ i^* \simeq L_\Sigma \circ i^! : \D(\sX) \to \D(\sZ)\inv,
  \end{equation}
  after applying the $\Sigma$-localization functor $L_\Sigma : \D(\sZ) \to \D(\sZ)\inv$ (see \corref{cor:shriekstar}).

  This functor $i^{!*}$ commutes with the four operations $f^*$, $f_*$, $f_!$, and $f^!$, as well as vanishing cycle sheaves.
  Thus in the situation of \thmref{thm:intro/equi}, and for a $G$-invariant function $w : X \to \A^1$, we have $i_0^{!*}\phi_w(\cF) \simeq L_\Sigma\phi_{w\circ i}(i^*\cF)$, where $i_0$ is the base change of $i$ to the zero locus of $w$ (\corref{cor:vancyctorus}).

  As we explain below, the functor $i^{!*}$ is closely related to Braden's hyperbolic localization functor in the torus-equivariant case, and to its analogue for $G=\mu_\ell$, the Smith restriction functor of Treumann.

\ssec{Torus actions: hyperbolic localization}

  Let $T$ be a split torus over $k$ and $\Sigma$ the set of nontrivial characters of $T$.
  Let $X$ be a quasi-separated algebraic space of finite type over $k$ with $T$-action, and $\lambda : \Gm \to T$ a cocharacter such that the induced $\Gm$-action on $X$ is étale-locally linearizable.
  We denote by $X^0$ the fixed locus of this $\Gm$-action, and by $X^+$ and $X^-$ its attractor and repeller, with the $T$-equivariant morphisms
  \[ X^0 \xleftarrow{\pi^\pm} X^\pm \xrightarrow{\iota^\pm} X \]
  and the closed immersion $i : X^0 \to X$ (see \cite[Thm.~A, Prop.~1.17]{RicharzHyperbolic}).
  The hyperbolic localization functors $\pi^+_!\iota^{+,*}$ and $\pi^-_*\iota^{-,!}$ of Braden \cite{BradenHyperbolic} fit in a chain
  \begin{equation}\label{eq:intro/chain}
    i^! \to \pi^-_*\iota^{-,!} \to \pi^+_!\iota^{+,*} \to i^*
  \end{equation}
  of functors $\D^T(X) \to \D^T(X^0)$.
  After $\Sigma$-localization, these four functors all collapse into $i^{!*}$ (see \corref{cor:hyploc}):

  \begin{thmX}[Hyperbolic localization]\label{thm:intro/hyploc}
    The above morphisms become invertible after $\Sigma$-localization.
    In particular, there are canonical isomorphisms of functors
    \[
      L_\Sigma \circ \pi^-_*\iota^{-,!} \simeq i^{!*} \simeq L_\Sigma \circ \pi^+_!\iota^{+,*}.
    \]
  \end{thmX}

  As an application, we obtain torus concentration for the perverse sheaves of cohomological Donaldson--Thomas theory.
  Let $X$ be a $(-1)$-shifted symplectic derived algebraic space, quasi-separated and of finite type over $\bC$, equipped with a $T$-action preserving the $(-1)$-shifted symplectic structure, and a $T$-equivariant orientation.
  We then have the Donaldson--Thomas sheaf $\phi_X \in \D^T(X)$ (see \cite[Thm.~4.8]{BBBBJ} and \cite[\S 4.2]{KPS}), where $\D$ denotes Betti sheaves of $\Q$-vector spaces, whose cohomology is the \emph{critical cohomology} of $X$.
  Let $\lambda : \Gm \to T$ be a cocharacter with fixed locus $i : X^0 \to X$.
  For each connected component $X^0_\alpha$ of $X^0$, let $\phi_{X^0_\alpha} \in \D^T(X^0_\alpha)$ be the Donaldson--Thomas sheaf of $[X^0_\alpha/T]$, and $I_\alpha \in \Z$ the value on $[X^0_\alpha/T]$ of the index function of \cite[Def.~6.6]{KPS}.
  Then we have (see \corref{cor:dt}):

  \begin{corX}[Concentration for Donaldson--Thomas sheaves]\label{cor:intro/dt}
    There are isomorphisms
    \begin{thmlist}
      \item\label{item:intro/dt/sheaf}
      $i^{!*}\phi_X \simeq \bigoplus_\alpha L_\Sigma(\phi_{X^0_\alpha})[I_\alpha]$ in $\D^T(X^0)\inv$;
      \item\label{item:intro/dt/coh}
      $\oH^*_T(X; \phi_X)\inv \simeq \bigoplus_\alpha \oH^{*+I_\alpha}_T(X^0_\alpha; \phi_{X^0_\alpha})\inv$.
    \end{thmlist}
  \end{corX}

  The hyperbolic localizations $\pi^+_!\iota^{+,*}\phi_X$ and $\pi^-_*\iota^{-,!}\phi_X$ are computed in \cite[Thm.~6.14]{Descombes} and \cite[Thm.~7.16]{KPS}, respectively, and \thmref{thm:intro/hyploc} implies that both collapse to $i^{!*}\phi_X$ after $\Sigma$-localization.
  Compare \cite[Rem.~8.3.5]{RSYZ}, where the analogous isomorphism for critical cohomology is obtained, after localization at the fraction field of $\oH^*_T(\pt)$, for moduli of spiked instantons via \cite[Thm.~6.2]{GoreskyKottwitzMacPherson}, as well as \cite[\S 2.2.5]{COZZ} for the case of global critical loci, where it is shown that the restriction to the fixed locus becomes surjective after such localization.

  For example, for $T = \{t_1t_2t_3 = 1\}$ acting on $\on{Hilb}^n(\A^3)$, \corref{cor:intro/dt} immediately yields that the localized critical cohomology is free over $\oH^*_T(\pt)\inv$ of rank the number of plane partitions of $n$ (\examref{exam:hilb}).
  Previously, this could be deduced from \cite[Rem.~10.10]{COZZII} and \cite[Cor.~6.7]{DavisonIntegrality}, as a consequence of Davison's purity theorem \cite[Cor.~6.6]{DavisonIntegrality} and the computation of the virtual motives of $\on{Hilb}^n(\A^3)$ in \cite[Thm.~2.7]{BBS}.

\ssec{Actions of \texorpdfstring{$\mu_\ell$}{mu-l}: Smith theory}

  Now take $G=\mu_\ell$ for $\ell$ a prime, and $\Sigma$ the set of nontrivial characters of $\mu_\ell$.
  In this case we will show that the functor $i^{!*}$ for the fixed locus $i : X^{\mu_\ell} \to X$ recovers, when restricted to constructible sheaves, the Smith restriction functor of Treumann \cite{TreumannSmith}.
  This makes precise the heuristic explained in \cite[Rem.~4.4]{TreumannSmith} that Smith restriction is the analogue for $\mu_\ell$-actions of hyperbolic localization for torus actions.

  Let $\D$ be Betti sheaves (over $k = \bC$) or étale sheaves (over $k$ algebraically closed with $\ell \in k^\times$) of $\Lambda$-modules, with $\Lambda$ as in \thmref{thm:smith}\itemref{item:smith/restr}, e.g. $\Lambda = \bF_\ell$, and let $X$ be a separated $k$-scheme of finite type.
  Write $\D^{\mu_\ell}_\mrm{ctf}(-) \sub \D^{\mu_\ell}(-)$ for the full subcategory of constructible objects of finite Tor-dimension, and $\D^{\mu_\ell}_\mrm{ctf}(-)_{\mu_\ell\text{-}\mrm{perf}}$ for those whose stalks are perfect complexes over the group algebras of the stabilizers.

  We will show that the ``Smith category'' of \cite[\S 4.2]{TreumannSmith} and \cite[\S 1.5]{RicheWilliamsonSmith} (for $\Lambda$ a field of characteristic $\ell$) embeds into the $\Sigma$-localized category via a fully faithful functor
  \[
    \on{Sm}(X^{\mu_\ell}) := \D^{\mu_\ell}_\mrm{ctf}(X^{\mu_\ell}) / \D^{\mu_\ell}_\mrm{ctf}(X^{\mu_\ell})_{\mu_\ell\text{-}\mrm{perf}} \hook \D^{\mu_\ell}(X^{\mu_\ell})\inv,
  \]
  and the functor $i^{!*} : \D^{\mu_\ell}(X) \to \D^{\mu_\ell}(X^{\mu_\ell})\inv$ restricts on $\D^{\mu_\ell}_\mrm{ctf}(X)$ to the ``Smith restriction'' functor
  \[
    i^{!*} : \D^{\mu_\ell}_\mrm{ctf}(X) \to \on{Sm}(X^{\mu_\ell}),
  \]
  of \cite[Def.~4.2]{TreumannSmith} and \cite[\S 1.5]{RicheWilliamsonSmith}.

  In particular, it then follows from $\Sigma$-localization theory that for every $\cF \in \D^{\mu_\ell}_\mrm{ctf}(X)$, the cofibre of $i^!\cF \to i^*\cF$ has perfect stalks (\thmref{thm:smith}\itemref{item:smith/restr}); this recovers \cite[Thm.~4.1]{TreumannSmith} for complex algebraic varieties, and its étale analogue \cite[Lem.~3.5]{RicheWilliamsonSmith}.

  At the level of cohomology, for $\mu_\ell^r$ acting on a separated algebraic space $X$ of finite type and $\cF$ the constant sheaf $\bF_\ell$, \thmref{thm:intro/equi} recovers the localization theorem of Quillen \cite[Thm.~4.2]{QuillenSpectrum} for elementary abelian $\ell$-groups (when $k = \bC$), and its étale analogue in \cite[Cor.~7.4]{IllusieZheng} (when $k$ is algebraically closed with $\ell\in k^\times$); see \ssecref{ssec:quillen}.

  For $\mu_\ell$-actions on stacks, we obtain the following generalization, where the fixed locus is replaced by the image of the homotopy fixed point stack $\sX^{h\mu_\ell} \to \sX$ (see \cite[Def.~A.14]{Virloc}).
  Let $\D$ be the derived \inftyCat of étale sheaves of $\Lambda$-modules, where $\Lambda$ is a $\Z/\ell^m$-algebra (with $\ell$ invertible in $k$), or of Betti sheaves of $\Lambda$-modules, where $\ell^m\Lambda = 0$ (with $k=\bC$).
  Then we have (see \corref{cor:smith hfix}):

  \begin{corX}[Smith concentration on stacks]\label{cor:intro/smithstack}
    Let $\sX$ be a quasi-compact algebraic stack over $k$ with separated diagonal and affine stabilizers.
    Suppose $\mu_\ell$ acts on $\sX$, and let $i : \sZ \to \sX$ be a $\mu_\ell$-equivariant closed immersion whose image contains that of $\sX^{h\mu_\ell} \to \sX$.
    Then the functor
    \[ i_* : \D^{\mu_\ell}(\sZ)\inv \to \D^{\mu_\ell}(\sX)\inv \]
    is an equivalence, and for every $\cF \in \D^{\mu_\ell}(\sX)$ the map
    \[ i^* : \oH^*_{\mu_\ell}(\sX; \cF)\inv \to \oH^*_{\mu_\ell}(\sZ; i^*\cF)\inv \]
    is invertible.
  \end{corX}

\ssec{Localization formulas}

  One standard consequence of the concentration theorem in torus-equivariant cohomology is the localization formula of Atiyah--Bott and Berline--Vergne \cite{AtiyahBott,BerlineVergne} for smooth compact manifolds, which expresses the inverse of the restriction to the fixed locus in terms of the Gysin map and the Euler class of the normal bundle.
  Let $T$ be a split torus over $k$ and $\Sigma$ the set of nontrivial characters of $T$.
  In the ``virtual'' case, we have (see \corref{cor:truck}):

  \begin{corX}[Cohomological virtual localization formula]\label{cor:intro/cohloc}
    Let $X$ be a quasi-compact algebraic stack with $T$-action, and let $Z \to X$ be a $T$-equivariant closed immersion such that $\St^T_X(x) \subsetneq T_x$ for every geometric point $x \in X \setminus Z$\footnote{%
      Here $T_x$ is the fibre $T \otimes_k k(x)$, and $\St^T_X(x) \sub T_x$ is the image of the map $\uAut_{[X/T]}(x) \to T_x$ induced by the $T$-action; see \cite[Def.~A.3]{Virloc}.
    }, and such that the $T$-action on $Z$ is trivial.
    Denote by $i : \sZ \to \sX$ the induced morphism between $\sZ := [Z/T] \simeq Z \times BT$ and $\sX := [X/T]$.
    Assume that $\cL_{\sZ/\sX} \in \Dperf^{[-2,-1]}(\sZ)$, $\cL_{\sZ/\sX}^\fix \simeq 0$, and $Z$ has finite stabilizers.
    Then for every $\cF \in \D^T(X)$, the isomorphisms
    \begin{align*}
      i^* &: \oH^*_T(X; \cF)\inv \to \oH^*_T(Z; i^*\cF)\inv,\\
      i_* &: \oH^*_T(Z; i^!\cF)\inv \to \oH^*_T(X; \cF)\inv
    \end{align*}
    have inverses given by
    \begin{equation}
      (i^*)^{-1} = i_!\left(- \cap e(\cN_{\sZ/\sX})^{-1}\right),
      \qquad
      (i_*)^{-1} = i^!(-) \cap e(\cN_{\sZ/\sX})^{-1},
    \end{equation}
    respectively, where $i_!$ and $i^!$ are the Gysin maps and $e(\cN_{\sZ/\sX})$ is the Euler class of the conormal complex $\cN_{\sZ/\sX} = \cL_{\sZ/\sX}[-1]$.
  \end{corX}

  For example, for $X$ a quasi-smooth derived algebraic space of finite type over $k$ with $T$-action, the homotopy fixed point space $X^{hT}$ is quasi-smooth, and $i : X^{hT} \to X$ is a closed immersion with classical truncation the fixed locus $X^T \sub X$ (see \cite[Def.~A.14, Prop.~A.21, Cor.~A.30]{Virloc}), and we get
  \[ (i^*)^{-1} = i_!\left(- \cap e(\cN_{X^{hT}/X})^{-1}\right) \]
  for the inverse of $i^* : \oH^*_T(X)\inv \to \oH^*_T(X^{hT})\inv$.
  Here the closed immersion $i$ is typically \emph{not} quasi-smooth: its conormal complex $\cN_{X^{hT}/X}$ is of Tor-amplitude $[0,1]$.
  The Gysin maps and the inverse Euler class are constructed by categorifying the constructions used to prove the virtual localization formula in Borel--Moore homology \cite[Thm.~D]{Virloc}.
  In particular, we prove a categorification of \corref{cor:intro/cohloc} which computes the inverse of the unit $\id \to i_*i^*$ in the $\Sigma$-localized category (\thmref{thm:catloc}).

\ssec{K-theory}

  The proofs also apply to sheaf theories that are not of Borel type.
  In fact, we will work throughout with $\D$ a \emph{topological weave} in the sense of \cite{Weaves} (roughly speaking, an abstract six functor formalism satisfying e.g. smooth and proper base change), which is oriented and satisfies the projective bundle formula (see \ssecref{ssec:weaves}).

  For the weave of genuine $\KGL$-motives of \cite{Sixstack}, in which cohomology on a quotient stack $[X/G]$ is the homotopy K-theory of $G$-equivariant vector bundles on $X$ (see \cite[\S 10.1]{Sixstack}), we obtain:

  \begin{corX}[Homotopy K-theory]\label{cor:intro/K}
    Let $X$ be a tame Deligne--Mumford stack of finite type over $k$ with an action of an affine algebraic group $G$ of multiplicative type.
    Let $\Sigma$ be a set of nontrivial rank one $G$-representations and let $Z \sub X$ be a $G$-invariant closed substack containing every geometric point $x$ of $X$ for which the restrictions $\rho|_{\St_x}$ are nontrivial for all $\rho\in\Sigma$.
    Then the inclusion $i : Z \hook X$ induces an isomorphism
    \[ i^* : \KH^G(X)[\Sigma^{-1}] \to \KH^G(Z)[\Sigma^{-1}], \]
    where $\rho \in \Sigma$ acts via multiplication by $\lambda_{-1}(\rho) = 1-[\rho]$.
  \end{corX}

  If $X$ is a quasi-projective scheme with linearizable action, $G$ may be any linearly reductive group.
  The Borel--Moore variant recovers Thomason's concentration theorem \cite[Thms.~2.1, 2.2]{ThomasonLefschetz} for equivariant G-theory, and the localization formula of \corref{cor:intro/cohloc} recovers the K-theoretic virtual localization formula of \cite[\S 5]{CiocanFontanineKapranov}, with no hypotheses on global embeddability or resolutions.

  \ssec{Conventions and notation}
  \label{ssec:convent}

    We fix an implicit noetherian base ring $k$.
    The term ``algebraic stack'' will mean \emph{derived algebraic stack locally of finite type over $k$} throughout the paper, where \emph{algebraic} is a synonym for \emph{$1$-Artin} (see e.g. \cite[\S 5.1]{ToenSimplicial}).
    Similarly, ``schemes'' and ``algebraic spaces'' will be derived and locally of finite type over $k$.
    We write $\Stk$ for the \inftyCat of algebraic stacks.

  \ssec{Acknowledgments}

    We are grateful to Hyeonjun Park for helpful discussions and suggestions, and to Pavel Safronov for asking about the relation with hyperbolic localization.

    The authors were supported by NSTC grant 110-2115-M-001-016-MY3 and Academia Sinica grant AS-CDA-112-M01 (A.K.) and EPSRC grant no. EP/R014604/1 (A.K. and C.R.).
    We would like to thank the Isaac Newton Institute for Mathematical Sciences, Cambridge, for support and hospitality during the programme KAH2 where we began work on this paper.
    The second author would like to thank the Max Planck Institute for funding and excellent working conditions during her stay there.

    This paper first appeared as a preprint in 2023.
    In revising it for publication, the authors used AI tools for proofreading, checking proofs, and finding references.

\section{Categorical setup}

  \ssec{Weaves}
  \label{ssec:weaves}

    Throughout the paper we will work in the context of an \emph{oriented topological weave}, which is a precise axiomatization of a sheaf theory with the six functor formalism.
    We give an informal account below and refer to \cite{Weaves,weavelisse} for the precise definitions.

    We fix a full subcategory $\bS \sub \Stk$ which is closed under finite coproducts and finite limits.
    We will work with weaves $\D$ on $\bS$ where representable morphisms are shriekable.
    Roughly speaking, this amounts to a collection of \inftyCats $\D(X)$ for every $X\in\bS$, equipped with a six functor formalism.

    \begin{enumcompress}
      \item[(1a)]\emph{Tensor:} $(-) \otimes (-) : \D(X) \times \D(X) \to \D(X)$,
      \item[(1b)]\emph{Internal Hom:} $\uHom(-, -) : \D(X)^\op \times \D(X) \to \D(X)$,
    \end{enumcompress}
    for every $X \in \bS$;

    \begin{enumcompress}
      \item[(2a)]\emph{$*$-Inverse image:} $f^* : \D(Y) \to \D(X)$,
      \item[(2b)]\emph{$*$-Direct image:} $f_* : \D(X) \to \D(Y)$,
    \end{enumcompress}
    for every morphism $f : X \to Y$;

    \begin{enumcompress}
      \item[(3a)]\emph{$!$-Direct image:} $f_! : \D(X) \to \D(Y)$,
      \item[(3b)]\emph{$!$-Inverse image:} $f^! : \D(Y) \to \D(X)$,
    \end{enumcompress}
    for every representable morphism $f : X \to Y$.\footnote{%
      Note that $f$ is necessarily locally of finite type, since all objects of $\Stk$ are locally of finite type over $k$.
    }

    The operation $\otimes$ is left adjoint to $\uHom$ (as bifunctors), $f^*$ is left adjoint to $f_*$, and $f_!$ is left adjoint to $f^!$.
    These operations are subject to various standard compatibilities and coherences: for example, $\otimes$ and $\uHom$ define a closed symmetric monoidal structure on $\D(X)$ (we write $\un_X \in \D(X)$ for the unit object); each of the operations $f^*$, $f_*$, $f_!$, $f^!$ is compatible with composition up to coherent homotopy; and $!$-direct image is compatible with $*$-inverse image (base change formula) and with $\otimes$ (projection formula) up to coherent homotopy.
    Moreover, for every proper representable morphism $f$ there is a canonical isomorphism $f_! \simeq f_*$ and for every smooth representable morphism there is a canonical isomorphism $f^! \simeq f^*(-) \otimes f^!(\un_Y)$ where $f^!(\un_Y)$ is $\otimes$-invertible (``Poincaré duality'').

    The condition that a weave $\D$ is \emph{topological} means that it satisfies the following:
    \begin{enumcompress}
      \item\label{item:weave/loc}
      \emph{Localization:}
      The \inftyCat $\D(\initial)$ is trivial (consists of zero objects).
      For any $X \in \bS$ and any closed-open decomposition $i : Z \to X$, $j : X\setminus Z \to X$, there is a canonical exact triangle of functors
      \begin{equation}\label{eq:loc}
        j_! j^* \to \id \to i_*i^*.
      \end{equation}

      \item\label{item:weave/htp}
      \emph{Homotopy invariance:}
      For any $X \in \bS$ and affine bundle $\pi : Y \to X$ (i.e., a torsor under a vector bundle on $X$), the unit morphism $\id \to \pi_*\pi^*$ is invertible.
    \end{enumcompress}

    We also assume that $\D$ satisfies Nisnevich descent (if all objects of $\bS$ are quasi-compact, this is automatic).
    These axioms imply that $\D(X)$ is stable for every $X$, and the localization axiom implies that the diagram
    \begin{equation*}
      \D(Z)
      \xrightarrow{i_*} \D(X)
      \xrightarrow{j^*} \D(U)
    \end{equation*}
    defines a (split) Verdier sequence of stable \inftyCats in the sense of \cite[App.~A]{HermitianII}.

    For a topological weave $\D$ we can write, for every smooth representable morphism $f : X \to Y$, $f^!(\un_Y) \simeq \un_X\vb{\Omega_f}$ where $\vb{\Omega_f}$ denotes the ``Thom twist'' by the relative cotangent sheaf $\Omega_f$.
    The choice of \emph{orientation} for $\D$ gives an identification
    $$(-)\vb{\Omega_f} \simeq (-)\vb{d} := (-)(d)[2d]$$
    where $d=\rk(\Omega_f)$ is the relative dimension of $f$, and $(d)$, resp. $[2d]$, indicates a Tate twist, resp. shift.
    Thus for an oriented topological weave, Poincaré duality reads $f^! \simeq f^*(-)\vb{d}$ for $f$ smooth representable of relative dimension $d$.

    For a locally free sheaf $\cE$ on $X\in\bS$ of rank $r$, the top Chern class may be regarded as a morphism $c_r(\cE) : \un_X \to \un_X\vb{r}$.
    For any object $\cF \in \D(X)$ this yields, by tensoring, a morphism
    \begin{equation}\label{eq:socht}
      c_r(\cE) : \cF \to \cF\vb{r}
    \end{equation}
    which in cohomology corresponds to cap product with $c_r(\cE)$.
    See \cite[Constr.~A.16]{KhanVirtual}.

    \begin{defn}\label{defn:PB}
      Given an oriented topological weave $\D$ on $\bS$, we consider the following condition:
      \begin{enumerate}
        \item[(PB1)]
        Let $X \in \bS$ and $\cE$ a locally free sheaf on $X$ of rank $r+1$, $r\ge 0$, with associated projective bundle $p : \P(\cE) \to X$.
        Then the natural transformations
        \begin{equation}
          \id
          \xrightarrow{\mrm{unit}} p_*p^*
          \xrightarrow{c_1(\cO(-1))^{\circ i}} p_*p^*\vb{i}
        \end{equation}
        induce a canonical isomorphism
        \begin{equation}\label{eq:0pgub1}
          \bigoplus_{0\le i\le r} \id\vb{-i}
          \to p_*p^*.
        \end{equation}

        \item[(PB2)]
        Let $X \in \bS$ and $\cE$ a locally free sheaf on $X$ of rank $r+1$, $r\ge 0$, with associated projective bundle $p : \P(\cE) \to X$.
        Then the natural transformation
        \[
          \mrm{unit} : \id_{\D(X)}
          \to p_*p^*
        \]
        admits a retract.
      \end{enumerate}
      By the projection formula for $p_*=p_!$, both conditions can be checked after evaluation on the unit object $\un_X$.
      The first condition is the projective bundle formula, and the second is the assertion that $p^* : \Ccoh(X) \to \Ccoh(\P(\cE))$ admits a retract.
      (PB1) implies (PB2): the projective bundle formula implies $p_!(p^*(-) \cap c_1(\cO(-1))^{\cup r}) \simeq \id_{\Ccoh(X)}$, where $p_! : \Ccoh(\P(\cE))\vb{r} \to \Ccoh(X)$ denotes the Gysin map in cohomology (notation as in \ssecref{ssec:cohcon}); see the discussion in \cite[Rem.~3.2.3]{DegliseOrientation}.
    \end{defn}

    \begin{rem}\label{rem:unebbed}
      For $\cE \simeq \cO_X^{\oplus r+1}$ it is easy to show (PB1) using the localization triangle, homotopy invariance, and induction on $r$.
      Thus (PB1) holds for any $\cE$ when $X$ is a scheme or Zariski-locally quasi-separated algebraic space, as we can reduce to the free case using Nisnevich descent.
      Note also that (PB2) holds for any $X \in \bS$ as soon as $\cE$ admits a surjection $\cE \twoheadrightarrow \cL$ with $\cL$ a line bundle, since we get a section $s$ of $p : \P(\cE) \to X$ giving rise to a retraction of $\id \to p_*p^*$.
    \end{rem}

    In general we will need to impose (PB2) as an additional axiom, although we expect it to hold for any oriented topological weave; in the next section, we will see that it holds for all the examples we consider.

  \ssec{Examples of weaves}

    \sssec{Lisse-extended weaves}
    \label{sssec:mashie}

      There are many examples of oriented topological weaves on the category of algebraic spaces:
      \begin{enumcompress}
        \item\emph{Betti sheaves:}
        If $k$ is a $\bC$-algebra and $\Lambda$ is a commutative ring, we may take $X \mapsto \D(X)$ sending $X$ to the derived \inftyCat $\on{D}(X(\bC); \Lambda)$ of sheaves of $\Lambda$-modules on the topological space $X(\bC)$.

        \item\emph{Étale sheaves (torsion coefficients):}
        If $n \in k^\times$ and $\Lambda$ is a commutative $\bZ/n\bZ$-algebra, we may take $X \mapsto \D(X)$ sending $X$ to the derived \inftyCat $\on{D}_\et(X; \Lambda)$ of sheaves of $\Lambda$-modules on the small étale site of $X$.

        \item\emph{Étale sheaves ($\ell$-adic coefficients):}
        If $\ell \in k^\times$, we may take $X \mapsto \D(X)$ sending $X$ to the $\ell$-adic derived \inftyCat $\on{D}_\et(X, \bZ_\ell)$ of sheaves on the small étale site of $X$, i.e., the limit $\on{D}_\et(X, \bZ/\ell^n\bZ)$ over $n>0$.

        \item\label{item:weave/motives}\emph{Motives:}
        Take $X \mapsto \D(X)$ sending $X$ to the \inftyCat $\on{DM}(X) := \on{D_{H\bZ}}(X)$ of modules over the integral motivic Eilenberg--MacLane spectrum $H\bZ_X$ as in \cite{Spitzweck}.
        (See \cite[Thm.~5.1]{CisinskiDegliseIntegral} and \cite[\S 14]{CisinskiDegliseBook} for comparisons with other constructions of derived categories of motives.)

        \item\emph{Cobordism motives:}
        Take $X\mapsto \D(X)$ sending $X$ to the \inftyCat $\MGLmod(X)$ of modules over Voevodsky's algebraic cobordism spectrum $\MGL_X$ (see e.g. \cite{EHKSY}).
      \end{enumcompress}

      Via the procedure of \emph{lisse extension} (see \cite{weavelisse}), every oriented topological weave $\D$ on algebraic spaces extends to a canonical oriented topological weave on $\bS = \Stk$.
      More precisely, for a stack $X$ we consider the limit
      \[ \D(X) := \lim_{(T,t)} \D(T) \]
      over the \inftyCat $\Lis_X$ of pairs $(T,t)$ where $T$ is a scheme and $t : T \to X$ is a smooth morphism, where the transition functors are $*$-inverse image.
      That this defines a topological weave on $\Stk$ is shown in \cite[\S 7.2]{weavelisse} and \cite{ChowdhuryArXiv} (see also \cite[App.~A]{KhanVirtual} and \cite{LiuZheng,LiuZhengAdic} for the case where $\D$ satisfies étale descent).
      Moreover, any orientation of $\D$ on algebraic spaces also lisse-extends to an orientation on $\Stk$.

      \begin{prop}\label{prop:corbiculum}
        Let $\D$ be an oriented topological weave on $\Stk$.
        Suppose that $\D$ is lisse-extended, i.e., that for every $X \in \Stk$ the family of functors $t^* : \D(X) \to \D(T)$ is jointly conservative as $t$ ranges over smooth morphisms $t : T \to X$ with $T$ a scheme.
        Then $\D$ satisfies condition~(PB1) of \defnref{defn:PB}.
      \end{prop}
      \begin{proof}
        We reduce to the case where $X$ is a scheme, which holds by \remref{rem:unebbed} (see also \cite[Thm.~2.3]{constack}).
      \end{proof}

    \sssec{Genuine \texorpdfstring{$\KGL$}{KGL}-motives}

      Let $\bS$ denote the \inftyCat of \emph{nicely} or \emph{linearly scalloped} algebraic stacks as in \cite[\S 2.3, App.~A]{Sixstack}.
      The class of nicely scalloped stacks includes tame algebraic stacks in the sense of \cite{AbramovichOlssonVistoli}, as well as quotients thereof by multiplicative type group actions.
      The class of linearly scalloped stacks includes quotients $[X/G]$ where $G$ is a linearly reductive affine group scheme acting linearly on a quasi-projective scheme $X$.

      If $X$ is nicely (resp. linearly) scalloped and $Y \to X$ is representable (resp. quasi-projective), then $Y$ is nicely (resp. linearly) scalloped.

      The main construction of \cite{Sixstack} provides a topological weave $\SH^\gen$ on $\bS$ sending $X$ to the \emph{genuine} stable motivic homotopy category $\SH^\gen(X)$.
      This restricts to $\SH$ on schemes, just like its lisse-extended version considered above, but does not agree for quotient stacks.
      We refer to \cite[\S 12]{Sixstack} for an extended discussion of this point.

      Taking modules over the genuine algebraic cobordism spectrum \cite[\S 10.2]{Sixstack} gives the oriented topological weave $\on{D_\MGL^\gen}$ of genuine $\MGL$-motives.
      Unfortunately, we do not know whether (PB1) or (PB2) holds in general for this weave\footnote{%
        To avoid confusion, let us reiterate that we do know this for lisse-extended cobordism motives.
      }.
      Therefore, we restrict our attention to modules over the genuine algebraic K-theory spectrum \cite[\S 10.1]{Sixstack}, which gives an oriented topological weave $\on{D_\KGL^\gen}$ of genuine $\KGL$-motives.

      \begin{prop}
        The oriented topological weave $\on{D_\KGL^\gen}$ satisfies condition~(PB1) of \defnref{defn:PB}.
      \end{prop}
      \begin{proof}
        In view of \cite[Prop.~3.7 and Thm.~10.7]{Sixstack} this amounts to the condition that for every nicely (resp. linearly) scalloped stack $X$, every locally free sheaf $\cE$ on $X$ of rank $r+1$, and every smooth representable (resp. quasi-projective) $Y \to X$, the maps of homotopy K-theory spectra
        \[
          p_Y^*(-) \cup \lambda_{-1}(\cO(-1)) : \KH(Y) \to \KH(\P(\cE|_Y),
        \]
        where $p_Y : \P_Y(\cE|_Y) \to Y$, induce an isomorphism
        \[
          \bigoplus_{0\le i\le r} \KH(Y) \to \KH(\P(\cE|_Y)).
        \]
        This follows from the corresponding formula for (non-homotopy) algebraic K-theory, see e.g. \cite[Cor.~3.6]{kblow}.
      \end{proof}

  \ssec{\texorpdfstring{$\Sigma$}{Sigma}-localization}
  \label{ssec:weaves/sigmalocal}

    From now on, $\D$ will be an oriented topological weave on $\bS$ as in \ssecref{ssec:weaves}, satisfying condition~(PB2) of \defnref{defn:PB}.

    Given a set of line bundles $\Sigma \sub \Pic(X)$, we denote by $\cS_\Sigma$ the set consisting of the morphisms \eqref{eq:socht}
    \[ c_1(\cL) : \cF \to \cF\vb{1} \]
    for all $\cL \in \Sigma$ and all $\cF \in \D(X)$.

    We consider the localization of $\D(X)$ at $\cS_\Sigma$ in the sense of \cite[\S 7.1]{CisinskiHCHA}.
    This is a (large\footnote{%
      Since $\cS_\Sigma$ is not essentially small, the localization $\D(X)\inv$ is formed in a larger universe, and is a priori not locally small.
      That it is in fact locally small follows from \lemref{lem:homloc} below.
    }) stable \inftyCat $\D(X)\inv$ with an essentially surjective and exact functor
    \begin{equation}\label{eq:fosse}
      L_\Sigma : \D(X) \to \D(X)\inv
    \end{equation}
    which inverts $\cS_\Sigma$ and satisfies the following universal property: for every stable \inftyCat $\cC$, restriction along \eqref{eq:fosse} determines an equivalence
    \begin{equation}\label{eq:overgreasy}
      \Fun_\mrm{ex}(\D(X)\inv, \cC)
      \to \Fun_{\mrm{ex},\Sigma}(\D(X), \cC),
    \end{equation}
    where the decoration ``$\mrm{ex}$'', resp. ``$\mrm{ex},\Sigma$'', indicates that we consider the full subcategory of exact functors, resp. exact functors inverting $\cS_\Sigma$.
    We say that a morphism in $\D(X)$ is a \emph{$\Sigma$-local equivalence} if it is inverted by $L_\Sigma$.

    The localization $\D(X)\inv$ can equivalently be described as the Verdier quotient by the full (stable) subcategory spanned by the cofibres of the morphisms in $\cS_\Sigma$ (see e.g. \cite[Prop.~A.1.5]{HermitianII}).
    Since the cofibre of $c_1(\cL) : \cF \to \cF\vb{1}$ is $\Cofib(c_1(\cL) : \un_X \to \un_X\vb{1}) \otimes \cF$, this is equivalently the Verdier quotient by the thick $\otimes$-ideal generated by the objects $\Cofib(c_1(\cL) : \un_X \to \un_X\vb{1})$ for $\cL \in \Sigma$ (see \cite[Lem.~A.1.8]{HermitianII}).

    Given any two functors $F$ and $G$ out of $\D(X)\inv$, there is a canonical isomorphism of \inftyGrpds of natural transformations
    \begin{equation}\label{eq:discouragingness}
      \on{Nat}(F, G)
      \to \on{Nat}(F\circ L_\Sigma, G\circ L_\Sigma),
    \end{equation}
    by taking mapping \inftyGrpds in \eqref{eq:overgreasy}.

    We will require some functoriality of the $\Sigma$-localization.
    Given a representable finite type morphism $f : Y \to X$, we also consider the localization of $\D(Y)$ with respect to the set $\cS_{f^*\Sigma}$, defined as above with $f^*\Sigma \sub \Pic(Y)$ in place of $\Sigma$:
    \[ L_\Sigma : \D(Y) \to \D(Y)\inv. \]
    Since $f^*$ is symmetric monoidal and $f^*(c_1(\cL)) \simeq c_1(f^*\cL)$ for every $\cL \in \Sigma$, $f^*$ sends $\cS_\Sigma$ to $\cS_{f^*\Sigma}$, so we obtain by the universal property an essentially unique exact functor
    \[ f^*_\Sigma : \D(X)\inv \to \D(Y)\inv \]
    such that $f^*_\Sigma \circ L_\Sigma \simeq L_\Sigma \circ f^*$.

    Moreover, $f_!$ sends $\cS_{f^*\Sigma}$ to $\cS_\Sigma$.
    Indeed, given $\cL \in \Sigma$ and $\cG \in \D(Y)$, the projection formula identifies $f_!(c_1(f^*\cL) : \cG \to \cG\vb{1})$ with $c_1(\cL) : f_!(\cG) \to f_!(\cG)\vb{1}$.
    Thus there exists an essentially unique exact functor
    \[ f_!^\Sigma : \D(Y)\inv \to \D(X)\inv \]
    such that $f_!^\Sigma \circ L_\Sigma \simeq L_\Sigma \circ f_!$.

    We discuss the respective right adjoints $f_*$ and $f^!$.
    The exchange transformations (see \cite[A.1.4, A.1.2]{weavelisse})
    \[ f_*(\cG) \otimes (-) \to f_*(\cG \otimes f^*(-)), \qquad f^*(-) \otimes f^!(\cF) \to f^!((-) \otimes \cF) \]
    are invertible on the $\otimes$-invertible objects $\un_X$ and $\un_X\vb{1}$, and identify
    \[
      f_*(c_1(f^*\cL) : \cG \to \cG\vb{1}), \qquad
      f^!(c_1(\cL) : \cF \to \cF\vb{1}),
    \]
    with $c_1(\cL) : f_*(\cG) \to f_*(\cG)\vb{1}$ and $c_1(f^*\cL) : f^!(\cF) \to f^!(\cF)\vb{1}$, respectively.
    This yields exact functors $f_*^\Sigma : \D(Y)\inv \to \D(X)\inv$ and $f^!_\Sigma : \D(X)\inv \to \D(Y)\inv$, right adjoint to $f^*_\Sigma$ and $f_!^\Sigma$, respectively.
    For example, to construct the unit transformation $\id \to f_*^\Sigma f^*_\Sigma$, use the isomorphism
    \begin{equation*}
      \on{Nat}(\id_{\D(X)\inv}, f_*^\Sigma f^*_\Sigma)
      \to \on{Nat}(L_\Sigma, f_*^\Sigma f^*_\Sigma L_\Sigma)
    \end{equation*}
    of \eqref{eq:discouragingness}.
    The desired natural transformation corresponds to the morphism $L_\Sigma \to L_\Sigma f_*f^* \simeq f_*^\Sigma f^*_\Sigma L_\Sigma$ induced by the unlocalized unit.

    Similarly, for every $n\in\Z$ the operation $\vb{n} : \D(Y) \to \D(Y)$ preserves $\cS_{f^*\Sigma}$, hence induces an operation on $\D(Y)\inv$ which we still denote by $\vb{n}$.

    By \eqref{eq:discouragingness}, the isomorphisms $f_! \simeq f_*$ for $f$ proper and $f^! \simeq f^*\vb{d}$ for $f$ smooth of relative dimension $d$ induce isomorphisms $f_!^\Sigma \simeq f_*^\Sigma$ and $f^!_\Sigma \simeq f^*_\Sigma\vb{d}$, respectively.

    \begin{lem}\label{lem:locseq}
      Let $i : Z \to X$ be a closed immersion and denote by $j : U \to X$ the complementary open immersion.
      Then there is a Verdier sequence of stable \inftyCats
      \begin{equation*}
        \D(Z)\inv
        \xrightarrow{i_*^\Sigma} \D(X)\inv
        \xrightarrow{j^*_\Sigma} \D(U)\inv.
      \end{equation*}
    \end{lem}
    \begin{proof}
      Since $\D$ satisfies the localization property, the claim holds before $\Sigma$-localization.
      It is clear that $j^*_\Sigma \circ i_*^\Sigma \simeq 0$ since this can be checked after applying $L_\Sigma$ on the right.
      Moreover, the localization triangle \eqref{eq:loc} gives rise, by applying $L_\Sigma$ on the left, to an exact triangle
      \[
        j_!^\Sigma j^*_\Sigma (L_\Sigma\cF)
        \to L_\Sigma\cF
        \to i_*^\Sigma i^*_\Sigma (L_\Sigma\cF)
      \]
      for every $\cF \in \D(X)$.
      The claim now follows formally.
    \end{proof}

    \begin{lem}\label{lem:nilp}
      Suppose $\Sigma \ne \emptyset$, and let $\cF \in \D(X)$.
      If $\pi_0\Maps_{\D(X)}(\cF, \cF\vb{n}) = 0$ for some $n > 0$, then $L_\Sigma(\cF) = 0$.
    \end{lem}
    \begin{proof}
      Take any $\cL \in \Sigma$.
      The morphism $c_1(\cL)^{\otimes n} \otimes \id_\cF : \cF \to \cF\vb{n}$ is a null-homotopic $\Sigma$-local equivalence.
    \end{proof}

    \begin{lem}\label{lem:strat}
      Suppose $X$ is quasi-compact, and let $(X_\alpha)_\alpha$ be a finite stratification of $X$ by locally closed substacks.
      Then for any $\cF \in \D(X)$, we have $L_\Sigma(\cF) = 0$ if and only if $L_\Sigma(\cF|_{X_\alpha}) = 0$ for all $\alpha$.
    \end{lem}
    \begin{proof}
      The ``only if'' direction holds since $L_\Sigma$ commutes with $*$-inverse image.
      For the converse, we show by noetherian induction that $L_\Sigma(\cF|_W) = 0$ for every closed substack $W \sub X$.
      If $W$ is nonempty, the generic point of an irreducible component of $W$ lies in some $X_\alpha$, so there exists some nonempty open $V \sub W$ contained in $X_\alpha$.
      Then $L_\Sigma(\cF|_V) = 0$, and since $L_\Sigma(\cF|_{W \setminus V}) = 0$ by induction hypothesis, the exact triangle in the proof of \lemref{lem:locseq} gives $L_\Sigma(\cF|_W) = 0$.
    \end{proof}

    \begin{rem}\label{rem:codicillary}
      If $\D(X)$ is presentable, one could also consider the \emph{cocontinuous} localization $\D(X)\llbracket\Sigma^{-1}\rrbracket$, i.e., the full subcategory of $\cS_\Sigma$-local objects in $\D(X)$ (see \cite[Rem.~7.7.10]{CisinskiHCHA}, \cite[\S 5.2.7]{LurieHTT}).
      However, for lisse-extended $\D$ we have $\D(X)\llbracket\Sigma^{-1}\rrbracket \simeq 0$ for every $X$ and nonempty $\Sigma$.
      Indeed, let $\cF \in \D(X)$ be $\cS_\Sigma$-local, $\cL \in \Sigma$, and $t : U \to X$ smooth of relative dimension $d$ with $U$ an affine scheme.
      Since $c_1(t^*\cL)$ is nilpotent, the projection formula shows that for every $\cG \in \D(U)$, some power of $c_1(\cL)$ acts on $t_!(\cG)\vb{d}$ by a null-homotopic $\Sigma$-local equivalence (as in the proof of \lemref{lem:nilp}).
      Since $t_!(-)\vb{d}$ is left adjoint to $t^!(-)\vb{-d} \simeq t^*$, it follows that $\Maps(\cG, t^*\cF) \simeq \Maps(t_!(\cG)\vb{d}, \cF) \simeq 0$ for all $\cG$, hence $t^*\cF \simeq 0$.
      Since $\D$ is lisse-extended, varying $t$ thus shows $\cF \simeq 0$.
    \end{rem}

  \ssec{Cohomology and mapping spectra}
  \label{ssec:cohcon}

    For any $X \in \bS$, the \emph{cohomology spectrum}, resp. with coefficients in $\cF \in \D(X)$, is the mapping spectrum
    \[
      \Ccoh(X) := \Ccoh(X; \un_X),
      \quad\text{resp.}~
      \Ccoh(X; \cF) := \Maps_{\D(X)}(\un_X, \cF).
    \]
    For an integer $n\in\Z$ we set
    \[
      \Ccoh(X; \cF)\vb{n} := \Ccoh(X; \cF\vb{n}) = \Maps_{\D(X)}(\un_X, \cF\vb{n})
    \]
    and $\Ccoh(X; \cF)\per := \bigoplus_{n\in\Z} \Ccoh(X; \cF)\vb{n}$.

    More generally, for $\cF, \cG \in \D(X)$ we consider the mapping spectrum
    \begin{equation}\label{eq:Mapsper}
      \Maps(\cF,\cG)\per := \bigoplus_{n\in\Z} \Maps_{\D(X)}(\cF, \cG\vb{n})
    \end{equation}
    and the $\Z$-graded abelian group $\on{Ext}^{\vb{*}}(\cF,\cG) := \pi_0\Maps(\cF,\cG)\per$.
    These are modules over the ring spectrum $\Ccoh(X)\per = \Maps(\un_X,\un_X)\per$, resp. the $\Z$-graded ring $\on{Ext}^{\vb{*}}(\un_X,\un_X)$.
    Note that $\Ccoh(X;\cF)\per = \Maps(\un_X,\cF)\per$.

    Given a subset $\Sigma \sub \Pic(X)$, let $S \sub \on{Ext}^{\vb{*}}(\un_X,\un_X)$ be the multiplicative subset generated by the elements $c_1(\cL)^{\otimes 2}$ for $\cL \in \Sigma$.
    Since they have even degree, these elements are central in $\pi_*\Ccoh(X)\per$ (see \cite[Prop.~3.3(d), Constr.~3.5]{BalmerSpectra}); in particular, $S$ satisfies the left Ore condition \cite[Def.~7.2.3.1]{LurieHA}.
    For a $\Ccoh(X)\per$-module spectrum $M$, we denote by
    $$M\inv \simeq \Ccoh(X)\per\inv \otimes_{\Ccoh(X)\per} M$$
    its localization at $S$ (see \cite[Rem.~7.2.3.18, Rem.~7.2.3.26]{LurieHA}), so that $\pi_*(M\inv) \simeq S^{-1}\pi_*M$ by \cite[Prop.~7.2.3.20]{LurieHA}.
    Since $c_1(\cL)$ acts invertibly on a module if and only if $c_1(\cL)^{\otimes 2}$ does, this is also the localization at the elements $c_1(\cL)$, $\cL \in \Sigma$.

    We define the $\Ccoh(X)\per\inv$-modules $\Maps_{\D(X)\inv}(\cF, \cG)\per$ as in \eqref{eq:Mapsper}, using the twists $\vb{n}$ on $\D(X)\inv$ (\ssecref{ssec:weaves/sigmalocal}).
    These can be computed by a calculus of fractions:

    \begin{lem}\label{lem:homloc}
      For $\cF, \cG \in \D(X)$, the functor $L_\Sigma$ induces an equivalence of $\Ccoh(X)\per\inv$-modules
      \[ \Maps_{\D(X)}(\cF,\cG)\per\inv \to \Maps_{\D(X)\inv}(L_\Sigma\cF, L_\Sigma\cG)\per. \]
      In particular, there is a canonical isomorphism
      \[\pi_0\Maps_{\D(X)\inv}(L_\Sigma\cF, L_\Sigma\cG) \simeq (S^{-1}\on{Ext}^{\vb{*}}(\cF,\cG))_0,\]
      where $(-)_0$ denotes the degree zero part, sending $f/s$, with $f \in \on{Ext}^{\vb{*}}(\cF,\cG)$ and $s \in S$ both of degree $n$, to $L_\Sigma(s\otimes\id_\cG)^{-1} \circ L_\Sigma(f)$.
    \end{lem}
    \begin{proof}
      The homotopy category $\on{h}(\D(X)\inv)$ is the Verdier quotient of $\on{h}\D(X)$ by the thick $\otimes$-ideal generated by $\Cofib(c_1(\cL) : \un_X \to \un_X\vb{1})$ for $\cL \in \Sigma$ (see \cite[Thm.~I.3.3]{NikolausScholze}), which by the octahedral axiom and \cite[Props.~2.13, 3.7]{BalmerSpectra} is also generated by the cofibres of the elements of $S$.
      Hence the second claim is \cite[Thm.~3.6]{BalmerSpectra}, applied to the $\otimes$-triangulated category $\on{h}\D(X)$ with the $\otimes$-invertible object $\un_X\vb{1}$.
      Applying it to $\cF$ and $\cG\vb{n}[-k]$ for all $n, k \in \Z$ shows that the map is invertible on homotopy groups.
    \end{proof}

    Let $\Mod_{\Ccoh(X)\per\inv}$ denote the \inftyCat of $\Ccoh(X)\per\inv$-module spectra.
    Given an open $U \sub X$, consider the exact functor $F : \D(U) \to \Mod_{\Ccoh(X)\per\inv}$ sending
    \begin{equation}\label{eq:murmurish}
      \cF \mapsto \Ccoh(X; j_!\cF)\per\inv
    \end{equation}
    where $j : U \to X$ denotes the inclusion.
    For every $\cF \in \D(U)$ and every $\cL \in \Sigma$, the projection formula identifies $F(c_1(j^*\cL) : \cF \to \cF\vb{1})$ with the morphism
    \[ c_1(\cL) : \Ccoh(X; j_!\cF)\per\inv \to \Ccoh(X; j_!\cF)\per\inv, \]
    which is invertible by $\Ccoh(X)\per\inv$-linearity, so $F$ inverts $\cS_{j^*\Sigma}$ and thus induces an essentially unique exact functor
    \begin{equation}\label{eq:bedull}
      F_\Sigma : \D(U)\inv \to \Mod_{\Ccoh(X)\per\inv}
    \end{equation}
    with $F_\Sigma\circ L_\Sigma \simeq F$.

\section{Concentration}

  \ssec{Acyclicity}

    \begin{thm}[Categorical acyclicity]\label{thm:acyclic}
      Let $X \in \bS$ with affine stabilizers, and $\Sigma \sub \Pic(X)$ a subset.
      Let $U \sub X$ be a quasi-compact open satisfying the following condition:
      \begin{thmlist}[label={\rm(L)},ref={(L)}]
        \item\label{cond:L}
        For every geometric point $x$ of $U$, there exists a line bundle $\cL(x) \in \Sigma$ whose restriction $\cL(x)|_{B\uAut_X(x)}$ is trivial.
      \end{thmlist}
      Then we have $\D(U)\inv = 0$.
    \end{thm}

    \begin{defn}
      If an open $U \sub X$ satisfies condition~\itemref{cond:L}, we say simply that $U$ is \emph{$\Sigma$-admissible}.
    \end{defn}

    \begin{rem}
      In fact, the proof of \thmref{thm:acyclic} will show that there exist $\cL_1, \ldots, \cL_N \in \Sigma$ such that  $\D(U)[\Sigma'^{-1}] = 0$ for the finite subset $\Sigma' = \{\cL_1, \ldots, \cL_N\} \sub \Sigma$.
    \end{rem}

    \begin{proof}[Proof of \thmref{thm:acyclic}]
      Replacing $X$ by $U$, $\Sigma$ by $j^*\Sigma$, and $j$ by the identity, we may assume $X=U$ is $\Sigma$-admissible and quasi-compact.

      We claim that there exist line bundles $\cL_1, \ldots, \cL_N \in \Sigma$ such that the morphism
      \[ s := c_1(\cL_1) \otimes \cdots \otimes c_1(\cL_N) : \un_X \to \un_X\vb{N} \]
      is null-homotopic.
      Given this, for every $\cF \in \D(X)$ the morphism $s \otimes \id_\cF : \cF \to \cF\vb{N}$ is null-homotopic; hence $L_\Sigma(s \otimes \id_\cF)$ is a null-homotopic isomorphism, so that $L_\Sigma(\cF) \simeq 0$.

      We begin with the following observations, where $s = \bigotimes_i c_1(\cL_i)$ for any $\cL_1, \ldots, \cL_N \in \Sigma$:
      \begin{enumerate}
        \item
        If $\pi : X' \to X$ is an affine bundle, then $\pi^*$ is fully faithful by homotopy invariance.
        Thus if $\pi^*(s)$ is null-homotopic, then $s$ is itself null-homotopic.

        \item
        If $\pi : X' \to X$ is a projective bundle, then $s$ is a retract of $\pi^*(s)$ by (PB2) (see \defnref{defn:PB}).
        Thus if $\pi^*(s)$ is null-homotopic, then $s$ is itself null-homotopic.

        \item\label{item:spongiosity}
        Let $i : Z \to X$ be a closed immersion with complementary open immersion $j : X \setminus Z \to X$.
        Suppose given $\cL'_1, \ldots, \cL'_{N'} \in \Sigma$ and $\cL''_1, \ldots, \cL''_{N''} \in \Sigma$ such that $i^*(s')$ and $j^*(s'')$ are null-homotopic, where $s' := c_1(\cL'_1) \otimes \cdots \otimes c_1(\cL'_{N'})$ and $s'' := c_1(\cL''_1) \otimes \cdots \otimes c_1(\cL''_{N''})$.
        Then $s' \otimes s''$ is also null-homotopic.
        Indeed, since $j^*(s'')$ is null-homotopic, the localization triangle $j_!j^*\un_X \to \un_X \to i_*i^*\un_X$ shows that $s''$ factors through the unit $\un_X \to i_*i^*\un_X$.
        Hence $s' \otimes s''$ factors through $s' \otimes \id_{i_*i^*\un_X}$, which is identified by the projection formula with $i_*(i^*(s'))$ and is thus null-homotopic.
      \end{enumerate}

      We now argue by noetherian induction on $X$.
      If $X$ is empty there is nothing to prove, so assume the claim holds for every closed substack $Z \subsetneq X$, with $\Sigma$ replaced by $i^*\Sigma$ where $i : Z \to X$ is the inclusion; note that $Z$ is again quasi-compact with affine stabilizers, and $i^*\Sigma$-admissible since $\uAut_Z(z) \simeq \uAut_X(z)$ for every geometric point $z$ of $Z$.
      By the induction hypothesis applied to $Z = X \setminus X'$ (with its reduced structure) and by \itemref{item:spongiosity}, it will suffice to find a nonempty open $X' \sub X$ and $\cL_1, \ldots, \cL_N \in \Sigma$ such that $s|_{X'}$ is null-homotopic.

      Since $X$ has affine stabilizers, there exists a nonempty open with a global quotient presentation (see \cite[Prop.~2.6]{HallRydhGroups}).
      We may therefore assume that $X = [X_0/G]$ where $G=\GL_n$ for some $n \geq 0$ and $X_0$ is a quasi-affine scheme of finite type over $k$ with $G$-action.

      Let $T$ (resp. $B$) denote the subgroup of diagonal matrices (resp. upper triangular matrices) in $G=\GL_n$.
      Since $B \sub G$ is a Borel subgroup, so that $G/B$ is a complete flag, the canonical morphism $[X_0/B] \to [X_0/G] = X$ factors as a sequence of iterated projective bundles.
      Moreover, the canonical morphism $[X_0/T] \to [X_0/B]$ is an affine bundle (see the proof of \cite[Thm.~1.13]{ThomasonAtiyahSegal}).
      By the above observations, it will thus suffice to prove the claim for $[X_0/T]$, with $\Sigma$ replaced by its inverse image.
      Note that $[X_0/T]$ is again quasi-compact with affine stabilizers, and admissible since its stabilizers are subgroups of those of $X$.
      Its closed substacks are of the form $[Y_0/T]$ for $T$-invariant closed subschemes $Y_0 \sub X_0$, which are again quasi-affine, so we may argue by noetherian induction on $X_0$ as above.
      That is, we may assume $X = [X_0/T]$, and it will suffice to find $\cL_1, \ldots, \cL_N \in \Sigma$ such that $s$ is null-homotopic after $*$-inverse image to some nonempty open of $X$.

      By \cite[Prop.~4.10, Rem.~4.11]{ThomasonComp} (where we may assume $X_0$ reduced by nil-invariance), there exists a nonempty $T$-invariant affine open $V\sub X_0$ and a diagonalizable subgroup $T'\sub T$ such that
      \[ [V/T] \simeq W \times BT' \]
      with $W$ an affine scheme.
      Shrinking $V$, we may assume that $W$ is connected.
      Choose a geometric point $v$ of $[V/T]$ and denote its image in $X$ by $x$.
      The $\Sigma$-admissibility of $X$ means that there exists an invertible sheaf $\cL \in \Sigma$ with $\cL|_{B\uAut_X(x)}$ trivial.
      Since $W$ is connected, \cite[Exp.~I, 4.7.3]{SGA3} gives an invertible sheaf $\cL'$ on $W$ and a character $\chi'$ of $T'$ such that
      \begin{equation} \label{eq:Kunneth}
        \cL|_{[V/T]} \simeq \cL' \boxtimes \cL_{\chi'}
      \end{equation}
      on $[V/T] \simeq W \times BT'$, where $\cL_{\chi'}$ is the line bundle associated to $\chi'$.
      The hypothesis on $\cL$ implies, after further restriction along $BT'_{k(v)} \simeq B\uAut_{[V/T]}(v) \to B\uAut_X(x)$, that $\chi'$ is trivial, i.e., $\cL|_{[V/T]} \simeq \pr_1^*(\cL')$.

      Since $W$ is a scheme, $c_1(\cL')$ is nilpotent (see e.g. the proof of \cite[Prop.~2.1.22(1)]{DegliseOrientation}), i.e., there exists an integer $N \gg 0$ such that $c_1(\cL')^{\otimes N} : \un_W \to \un_W\vb{N}$ is null-homotopic.
      Hence $s := c_1(\cL)^{\otimes N}$ becomes null-homotopic after $*$-inverse image to $[V/T]$.
    \end{proof}

    Let us also record the following variant of \thmref{thm:acyclic}:

    \begin{varnt}\label{varnt:acyclic ell}
      Let $\ell$ be a prime and $m \ge 1$, and let $\D$ be the étale weave with coefficients in a $\Z/\ell^m$-algebra $\Lambda$ (with $\ell$ invertible on $k$), or the Betti weave with coefficients in a ring $\Lambda$ with $\ell^m\Lambda = 0$ (for $k=\bC$).
      Then \thmref{thm:acyclic} (and hence \thmref{thm:concat} and \corref{cor:concoh} below), hold with \itemref{cond:L} replaced by the weaker condition:
      \begin{thmlist}[label={\rm(L$_\ell$)},ref={(L$_\ell$)}]
        \item\label{cond:Lell}
        For every geometric point $x$ of $U$, there exists $\cL(x) \in \Sigma$ whose restriction to $BH$ is trivial for every closed subgroup $H \sub \uAut_X(x)$ with $H \simeq \mu_{\ell,k(x)}$.
      \end{thmlist}
      To see this, one argues exactly as in the proof above up to \eqref{eq:Kunneth}.
      Closed subgroups $H \simeq \mu_{\ell,k(v)}$ of $T'_{k(v)}$ correspond to surjections $\phi : X(T') \to \Z/\ell$, and $\cL_{\chi'}|_{BH}$ is trivial if and only if $\phi(\chi') = 0$ (see the proof of \corref{cor:equiconcat diag}).
      Since $X(T')$ is finitely generated, \itemref{cond:Lell} thus gives $\chi' = \ell\psi$ for a character $\psi$ of $T'$.
      Then \eqref{eq:Kunneth} gives $c_1(\cL|_{[V/T]}) = \pr_1^*c_1(\cL') + \ell \cdot \pr_2^*c_1(\cL_\psi)$ in $\oH^{2*}([V/T];\Lambda(*))$.
      Since $c_1(\cL')^N = 0$ and $\ell^m\Lambda = 0$, we have $c_1(\cL|_{[V/T]})^{N+m-1} = 0$ and we conclude as above.
    \end{varnt}

    Now if $\cF \in \D(U)$ such that $L_\Sigma(\cF)\simeq 0$, then also $F(\cF) \simeq F_\Sigma(L_\Sigma(\cF)) \simeq 0$.
    In particular, categorical acyclicity (\thmref{thm:acyclic}) yields:

    \begin{cor}[Cohomological acyclicity]\label{cor:cohacyclic}
      Let $X \in \Stk$ with affine stabilizers and $\Sigma \sub \Pic(X)$ a subset of line bundles.
      If $U \sub X$ is a $\Sigma$-admissible quasi-compact open, then we have $$\Ccoh(X; j_!\cF)\per\inv \simeq 0$$
      for all $\cF \in \D(U)$, where $j : U \to X$ denotes the inclusion.
      In particular, we have $\Ccoh(U; \cF)\per\inv \simeq 0$.
    \end{cor}

    For the second statement, apply the first with $X$ replaced by $U$ (and $\Sigma$ by $j^*\Sigma$, $j$ by the identity).

  \ssec{Concentration}

    \begin{thm}[Categorical concentration]\label{thm:concat}
      Let $X \in \Stk$ with affine stabilizers.
      Let $\Sigma \sub \Pic(X)$ be a subset of line bundles.
      For any closed immersion $i: Z \to X$ with quasi-compact and $\Sigma$-admissible complement $X \setminus Z$, the functor
      \[ i_*^\Sigma: \D(Z)\inv \to \D(X)\inv \]
      is an equivalence.
    \end{thm}
    \begin{proof}
      We have $\D(X \setminus Z)\inv = 0$ by acyclicity (\thmref{thm:acyclic}), so the claim follows from \lemref{lem:locseq}.
    \end{proof}

    \begin{cor}[Cohomological concentration]\label{cor:concoh}
      Let $X\in\Stk$ with affine stabilizers and $\Sigma \sub \Pic(X)$ a subset of line bundles.
      Let $i : Z \to X$ be a closed immersion such that $X \setminus Z$ is $\Sigma$-admissible and quasi-compact.
      Then for every $\cF \in \D(X)$, the unit $\cF \to i_*i^*\cF$ and counit $i_*i^!\cF \to \cF$ induce isomorphisms
      \begin{align}
        \label{eq:anarchy1}
        \Ccoh(X; \cF)\per\inv \to \Ccoh(Z; i^*\cF)\per\inv\\
        \label{eq:anarchy}
        \Ccoh(Z; i^!\cF)\per\inv \to \Ccoh(X; \cF)\per\inv.
      \end{align}
    \end{cor}

    \begin{proof}\leavevmode
      Consider the localization triangle
      \[
        \Ccoh(X; j_!j^*\cF)\per
        \to \Ccoh(X; \cF)\per
        \to \Ccoh(X; i_*i^*\cF)\per
      \]
      where $j$ is the inclusion of $U = X\setminus Z$.
      It remains exact after applying the exact functor $(-)\inv$.
      Since $U$ is $\Sigma$-admissible and quasi-compact, the left-most term vanishes after localization by \corref{cor:cohacyclic}.
      This shows the first isomorphism \eqref{eq:anarchy1}.

      Similarly, we have the localization triangle
      \[
        \Ccoh(X; i_*i^!\cF)\per
        \to \Ccoh(X; \cF)\per
        \to \Ccoh(X; j_*j^*\cF)\per
      \]
      where the right-most term $\Ccoh(X; j_*j^*\cF)\per \simeq \Ccoh(U; j^*\cF)\per$ vanishes after localization by \corref{cor:cohacyclic}.
    \end{proof}

    \begin{rem}\label{rem:keelhale}
      The second isomorphism \eqref{eq:anarchy} recovers concentration in Borel--Moore homology (as in \cite[Cor.~3.7]{constack}): taking $\cF = a_X^!(\Lambda)$ where $a_X : X \to \Spec(k)$ is the projection and $\Lambda \in \D(\Spec(k))$, we get the isomorphism
      \[
        i_* : \CBM(Z; \Lambda)\per\inv \to \CBM(X; \Lambda)\per\inv
      \]
      where $\CBM(X; \Lambda) := \Ccoh(X; a_X^!(\Lambda))$ and similarly for $Z$.
    \end{rem}

    \begin{varnt}\label{varnt:tempre}
      We can replace all occurrences of $\Ccoh(X; -) = \Maps(\un_X, -)$ above by $a_{X,*}(-)$, where $a_X : X \to \Spec(k)$ is the projection.
      In other words, we can stop short of taking derived global sections $\RGamma(\Spec(k), -) : \D(\Spec(k)) \to \Spt$.
      Thus for example when $\D$ is the weave of motives (or cobordism motives), we get isomorphisms
      \begin{align*}
        i^* : \on{M}(X)\inv &\to \on{M}(Z)\inv,\\
        i_* : \on{M}^\BM(Z)\inv &\to \on{M}^\BM(X)\inv
      \end{align*}
      under the assumptions of \corref{cor:concoh}, where $\on{M}(X) := a_{X,*}a_X^*(\un)$ and $\on{M}^\BM(X) := a_{X,*}a_X^!(\un)$, and similarly for $Z$.
    \end{varnt}

  \ssec{The functor \texorpdfstring{$i^{!*}$}{i^!*}}

    \begin{cor}\label{cor:shriekstar}
      Let $X$, $\Sigma$ and $i : Z \to X$ be as in \thmref{thm:concat}.
      Then for every $\cF \in \D(X)$, the canonical morphism
      \[ i^!\cF \simeq i^*i_*i^!\cF \xrightarrow{\mrm{counit}} i^*\cF \]
      is a $\Sigma$-local equivalence, where the first isomorphism is inverse to the counit $i^*i_* \to \id$.
    \end{cor}
    \begin{proof}
      Applying $i_*$ to $i^!\cF \to i^*\cF$ gives the composite of the counit $i_*i^!\cF \to \cF$ and the unit $\cF \to i_*i^*\cF$.
      Their cofibres are killed by $L_\Sigma$ (\thmref{thm:acyclic}), whence the claim.
    \end{proof}

    In view of \corref{cor:shriekstar}, we obtain the \emph{$\Sigma$-localized restriction} functor
    \begin{equation}\label{eq:shriekstar}
     	i^{!*} := L_\Sigma \circ i^* \simeq L_\Sigma \circ i^! : \D(X) \to \D(Z)\inv.
    \end{equation}
    As we will see below, this functor is closely related to hyperbolic localization, in the case of quotients by torus actions (see \secref{sec:torus} below), and to the \emph{Smith restriction} functor of \cite[Def.~4.2, Rem.~4.3]{TreumannSmith} in the case of $\mu_\ell^r$-actions (see \secref{sec:smith} below).
    It is immediate that $i^{!*}$ commutes with the four operations $f^*$, $f_*$, $f_!$, $f^!$ (compare \cite[Thms.~4.2, 4.4]{TreumannSmith} and \cite[Prop.~6.2]{RicheWilliamsonSmith}).

  \ssec{Concentration on mapping spectra}

    \begin{cor}\label{cor:extloc}
      Let $X$, $\Sigma$ and $i : Z \to X$ be as in \thmref{thm:concat}, and $\cF, \cG \in \D(X)$.
      Then $i^{!*}$ induces an equivalence of $\Ccoh(X)\per\inv$-modules (regarding $\Ccoh(Z)\per$-modules as $\Ccoh(X)\per$-modules via $i^*$)
      \[\Maps(\cF,\cG)\per\inv \simeq \Maps_{\D(Z)\inv}(i^{!*}\cF, i^{!*}\cG)\per.\]
      In particular, the map $\Maps(i^*\cF, i^!\cG)\per\inv \to \Maps(\cF,\cG)\per\inv$ induced by the unit $\cF \to i_*i^*\cF$ and the counit $i_*i^!\cG \to \cG$ is invertible.
    \end{cor}
    \begin{proof}
      By \lemref{lem:homloc}, the map is induced by $i^*_\Sigma : \D(X)\inv \to \D(Z)\inv$, which is inverse to the equivalence $i_*^\Sigma$ (\thmref{thm:concat}).
      Similarly, the second claim follows from the full faithfulness of $i_*^\Sigma$ and the fact that $L_\Sigma$ inverts the unit $\cF \to i_*i^*\cF$ and the counit $i_*i^!\cG \to \cG$ (see the proof of \corref{cor:shriekstar}).
    \end{proof}

    For $\cF = \un_X$, the second claim recovers \eqref{eq:anarchy}, and the first (together with \lemref{lem:homloc} for $Z$) recovers \eqref{eq:anarchy1}.
    On $\pi_0$, this gives isomorphisms
    \begin{gather*}
      S^{-1}\on{Ext}^{\vb{*}}(\cF,\cG) \simeq S^{-1}\on{Ext}^{\vb{*}}(i^*\cF,i^*\cG) \simeq S^{-1}\on{Ext}^{\vb{*}}(i^!\cF,i^!\cG),\\
      S^{-1}\on{Ext}^{\vb{*}}(i^*\cF,i^!\cG) \simeq S^{-1}\on{Ext}^{\vb{*}}(\cF,\cG),
    \end{gather*}
    where $\on{Ext}^{\vb{*}}$ and $S$ are as in \ssecref{ssec:cohcon}.
    For the Betti weave (where $\vb{1}\simeq[2]$), we recover the same for the usual graded groups $\on{Ext}^*$ by taking $\pi_0$ and $\pi_{-1}$ together.
    In that form, for a torus acting on a complex variety and $\bC$-coefficients, these isomorphisms appear in \cite[(3.2.1)--(3.2.2)]{BravermanFinkelbergNakajima} and, for particular $\cF$ and $\cG$, in \cite[\S 4.7]{EvensMirkovic}.

  \ssec{Equivariant concentration}
  \label{ssec:diag}

    Let $G$ be an affine fppf group scheme over $k$.
    For $X \in \Stk$ with $G$-action, we write $\D^G(X) := \D([X/G])$ and $\Ccoh_G(X; \cF) := \Ccoh([X/G]; \cF)$.
    For a $G$-equivariant morphism, we use the same notation for the induced morphism of quotient stacks and its operations on $\D^G(-)$.
    Given $\Sigma \sub \Pic(BG)$, we also write $\Sigma$ for the set $f^*\Sigma\sub \Pic(\sX)$ for any $f : \sX \to BG$.

    \begin{thm}[Equivariant concentration]\label{thm:equiconcat}
      Let $X \in \Stk$ be quasi-compact with affine stabilizers and $G$-action, $\Sigma \sub \Pic(BG)$ a subset, and $Z \to X$ a $G$-equivariant closed immersion such that for every geometric point $x \in X \setminus Z$, some $\cL \in \Sigma$ has trivial restriction to $B\St^G_X(x)$.\footnote{\label{fn:equiconcat}%
        Here $G_x$ is the fibre $G \otimes_k k(x)$ and $\St^G_X(x)$ is the \emph{$G$-stabilizer} at a geometric point $x \in X$, defined as the image of the map $\alpha: \uAut_{[X/G]}(x) \to G_x$ induced by the $G$-action.
        See \cite[Def.~A.3]{Virloc}.
      }
      Denote by $i : Z \to X$ the inclusion.
      Then $[(X\setminus Z)/G]$ is $\Sigma$-admissible.
      In particular, the functor
      \[i_*: \D^G(Z)\inv \to \D^G(X)\inv \]
      is an equivalence, and for every $\cF \in \D^G(X)$ the map
      \[ i^* : \Ccoh_G(X; \cF)\per\inv \to \Ccoh_G(Z; i^*\cF)\per\inv \]
      is invertible.
    \end{thm}
    \begin{proof}
      Write $\sX := [X/G]$, and let $x$ be a geometric point of $X \setminus Z$.
      The group $\uAut_\sX(x)$ is an extension of $\St^G_X(x) \sub G_x$ by $\uAut_X(x)$ (see \cite[Def.~A.3]{Virloc}), so $\sX$ has affine stabilizers.
      Since $\St^G_X(x)$ is the image of $\alpha : \uAut_\sX(x) \to G_x$, the composite $B\uAut_\sX(x) \to \sX \to BG$ factors through $B\St^G_X(x)$.
      Hence an $\cL \in \Sigma$ with trivial restriction to $B\St^G_X(x)$ also has trivial restriction to $B\uAut_\sX(x)$, so $[(X\setminus Z)/G]$ is $\Sigma$-admissible (cf. \cite[Def.~3.13, Cor.~3.14]{constack}).
      The remaining claims follow from \thmref{thm:concat} and \corref{cor:concoh}.
    \end{proof}

    In the situation of \thmref{thm:equiconcat}, we thus have the functor
    \begin{equation}\label{eq:i^!* G}
     	i^{!*} : \D^G(X) \to \D^G(Z)\inv
    \end{equation}
    of \eqref{eq:shriekstar}.

    Suppose now that $G$ is diagonalizable and that $k$ has no nontrivial idempotents, and denote by $\Sigma \sub \Pic(BG)$ the set of line bundles associated with nontrivial characters of $G$.

    \begin{cor}\label{cor:equiconcat diag}
      Let $H \sub G$ be a closed subgroup, and $\Sigma_H \sub \Sigma$ the subset of line bundles associated with characters of $G$ that are nontrivial on $H$.
      Let $X \in \Stk$ be quasi-compact with affine stabilizers and $G$-action, and $Z \to X$ a $G$-equivariant closed immersion such that for every geometric point $x \in X \setminus Z$ we have $H_x \not\sub \St^G_X(x)$, where $H_x := H \otimes_k k(x)$.
      Then the conclusions of \thmref{thm:equiconcat} hold with $\Sigma_H$ in place of $\Sigma$.
      For $H = G$, we have $\Sigma_G = \Sigma$, and the condition reads $\St^G_X(x) \subsetneq G_x$.
    \end{cor}
    \begin{proof}
      Let $x$ be a geometric point of $X \setminus Z$.
      Since $\St^G_X(x)$ is affine, line bundles on $B\St^G_X(x)$ correspond to its characters (e.g. by \cite[Exp.~I, Prop.~4.7.2]{SGA3}), so the line bundle associated with a character $\chi$ of $G$ has trivial restriction to $B\St^G_X(x)$ if and only if $\chi$ is trivial on $\St^G_X(x)$.
      Write $G = D(M) := \Spec(k[M])$ for an abelian group $M$, so that $G_x = D_{k(x)}(M)$.
      By \cite[Exp.~IX, Prop.~8.1]{SGA3} and \cite[Exp.~VIII, Cor.~3.2(b), Cor.~3.3]{SGA3}, every closed subgroup $K \sub G_x$ is equal to $D_{k(x)}(M/K^\perp)$, where $K^\perp \sub M$ is the subgroup of characters trivial on $K$; in particular, $K \sub K'$ if and only if $K'^\perp \sub K^\perp$.
      Since $H_x \not\sub \St^G_X(x)$, we get $\St^G_X(x)^\perp \not\sub H_x^\perp$, i.e., there is a character of $G$ that is trivial on $\St^G_X(x)$ but not on $H$.
      Thus $\Sigma_H$ satisfies the condition of \thmref{thm:equiconcat}.
    \end{proof}

    For example, if $X$ is an algebraic space and $G$ is a torus, we may take $Z = X^G$ (see \cite[Rem.~A.10, Prop.~A.21]{Virloc}).
    More generally:

    \begin{rem}\label{rem:equiconcat subgroup}
      For example, if $X$ is an algebraic space, we may take $Z = X^H$ whenever $X^H \to X$ is a closed immersion: indeed, the $H$-stabilizer of $x$ is $H_x \cap \St^G_X(x)$ (see \cite[Rem.~A.5]{Virloc}), so $x \notin X^H$ if and only if $H_x \not\sub \St^G_X(x)$ (see \cite[Rem.~A.10]{Virloc}).
      This holds if $H$ is flat with connected fibres (see \cite[Prop.~A.21]{Virloc}), or if $H$ is a finite constant group and $X$ is separated.
    \end{rem}

  \ssec{Example: vanishing cycles}

    We next derive concentration for vanishing cycles.
    Suppose that $\D$ is equipped with vanishing cycle functors $\phi_w : \D(X) \to \D(X_0)$, for every stack $X \in \bS$ and function $w : X \to \A^1$ with zero locus $X_0$, such that
    \begin{enumcompress}
      \item\label{item:vanlin} $\phi_w(\cF \otimes \cP) \simeq \phi_w(\cF) \otimes \cP|_{X_0}$, naturally in $\cF \in \D(X)$ and $\otimes$-invertible $\cP \in \D(X)$;
      \item\label{item:vanpbc} $\phi_w \circ i_* \simeq i_{0,*} \circ \phi_{w \circ i}$ for every closed immersion $i : Z \to X$, with base change $i_0 : Z_0 \to X_0$.
    \end{enumcompress}
    For example, this holds for Betti sheaves, for étale sheaves with torsion coefficients over a separably closed field, or for motivic sheaves with rational coefficients over a base of characteristic zero (see e.g. \cite[App.~A]{HoskinsPepinLehalleur}).

    \begin{cor}\label{cor:vancyc}
      Let $X$, $\Sigma$ and $i : Z \to X$ be as in \thmref{thm:concat}, and $w : X \to \A^1$ a function.
      Then for every $\cF \in \D(X)$, the morphism
      \[ \phi_w(\cF) \xrightarrow{\mrm{unit}} \phi_w(i_*i^*\cF) \simeq i_{0,*}\phi_{w\circ i}(i^*\cF) \]
      is a $\Sigma$-local equivalence, where $i_0 : Z_0 \to X_0$ is the base change of $i$.
      In other words, we have $i_0^{!*}\phi_w(\cF) \simeq L_\Sigma\phi_{w\circ i}(i^*\cF)$.
    \end{cor}
    \begin{proof}
      By \itemref{item:vanlin}, $\phi_w$ sends $c_1(\cL) \otimes \id_\cF$ to $c_1(\cL|_{X_0}) \otimes \id_{\phi_w(\cF)}$, so it preserves $\Sigma$-local equivalences.
      The unit $\cF \to i_*i^*\cF$ is a $\Sigma$-local equivalence, as in the proof of \corref{cor:shriekstar}.

      For the second statement, note that $i_0$ satisfies the assumptions of \thmref{thm:concat}, since $X_0 \setminus Z_0$ is closed in $X \setminus Z$, so $i_0^{!*} := L_\Sigma \circ i_0^*$ is defined.
      Since $i_0^*$ preserves $\Sigma$-local equivalences, using $i_0^*i_{0,*} \simeq \id$ gives the claim.
    \end{proof}

    We also obtain an equivariant version of \corref{cor:vancyc}.
    For this suppose $\D$ is equipped with vanishing cycle functors $\phi_w$ satisfying the conditions \itemref{item:vanlin} and \itemref{item:vanpbc} stated before \corref{cor:vancyc}.

    \begin{cor}\label{cor:vancyctorus}
      Let $X$, $\Sigma$, $Z$ and $i : Z \to X$ be as in \thmref{thm:equiconcat}.
      Let $w : X \to \A^1$ be a $G$-invariant function, and write also $w : [X/G] \to \A^1$ for the induced function.
      Then for every $\cF \in \D^G(X)$, the morphism
      \[ \phi_w(\cF) \xrightarrow{\mrm{unit}} \phi_w(i_*i^*\cF) \simeq i_{0,*}\phi_{w\circ i}(i^*\cF) \]
      is a $\Sigma$-local equivalence.
      In other words, we have $i_0^{!*}\phi_w(\cF) \simeq L_\Sigma\phi_{w\circ i}(i^*\cF)$.
    \end{cor}

\section{Hyperbolic localization for torus actions}
\label{sec:torus}

  In this section, we take the diagonalizable group scheme $G$ of \ssecref{ssec:diag} to be a torus, which we denote by $T$.

  \ssec{Comparison with hyperbolic localization}

    We now explain the relation between the functor $i^{!*}$ and Braden's hyperbolic localization \cite{BradenHyperbolic}.

    Let $X$ be a quasi-separated algebraic space of finite type over $k$ with an action of $T$, and $\lambda : \Gm \to T$ a cocharacter such that the induced $\Gm$-action on $X$ is étale-locally linearizable.
    Let $X^0$ be the fixed locus of this $\Gm$-action, and $X^+$ and $X^-$ the attractor and repeller (= attractor for $\lambda^{-1}$), with the $T$-equivariant morphisms
    \[ X^0 \xleftarrow{\pi^\pm} X^\pm \xrightarrow{\iota^\pm} X, \]
    the closed immersions $s^\pm : X^0 \to X^\pm$, and the closed immersion $i = \iota^\pm \circ s^\pm : X^0 \to X$ (see \cite[Thm.~A, Prop.~1.17]{RicharzHyperbolic}).

    The hyperbolic localization functors $\pi^+_!\iota^{+,*}, \pi^-_*\iota^{-,!} : \D^T(X) \to \D^T(X^0)$ fit in a chain
    \[
      i^! \to \pi^-_*\iota^{-,!} \to \pi^+_!\iota^{+,*} \to i^*.
    \]
    For the Betti and étale weaves, Braden's theorem asserts that the middle arrow is invertible on bounded constructible objects, see \cite[Thm.~1]{BradenHyperbolic}, \cite[\S 0.4, Thm.~3.1.6]{DrinfeldGaitsgoryBraden}, \cite[Thm.~B]{RicharzHyperbolic}.

    Here we show that, for an arbitrary topological weave, after $\Sigma$-localization, all functors in this chain collapse to the functor $i^{!*}$ of \eqref{eq:i^!* G}\footnote{%
      where $i$ satisfies the assumptions of \corref{cor:equiconcat diag} since the $T$-stabilizer of any geometric point of $X \setminus X^0$ does not contain the image of $\lambda$
    }:

    \begin{cor}\label{cor:hyploc}
      The canonical morphisms $i^! \to \pi^-_*\iota^{-,!}$ and $\pi^+_!\iota^{+,*} \to i^*$ are $\Sigma$-local equivalences.
      In particular, there are canonical isomorphisms of functors $\D^T(X) \to \D^T(X^0)\inv$
      \[
        L_\Sigma \circ \pi^-_*\iota^{-,!} \simeq i^{!*} \simeq L_\Sigma \circ \pi^+_!\iota^{+,*}.
      \]
    \end{cor}
    \begin{proof}
      The morphisms in question are induced by the unit $\id \to s^+_*s^{+,*}$ and the counit $s^-_*s^{-,!} \to \id$:
      \[
        i^! \simeq \pi^-_*s^-_*s^{-,!}\iota^{-,!} \to \pi^-_*\iota^{-,!},
        \qquad
        \pi^+_!\iota^{+,*} \to \pi^+_!s^+_*s^{+,*}\iota^{+,*} \simeq i^*.
      \]
      Since $\pi^+_!$ and $\pi^-_*$ preserve $\Sigma$-local equivalences, it will suffice to show that this unit and counit are $\Sigma$-local equivalences.
      For this, apply \corref{cor:equiconcat diag} to $s^\pm$: $X^\pm$ is quasi-compact (see \cite[Thm.~A(iii)]{RicharzHyperbolic}), and the $T$-stabilizer of any geometric point $x$ of $X^\pm \setminus X^0$ does not contain the image of $\lambda$ (since $X^0$ is the $\Gm$-fixed locus of $X^\pm$).
    \end{proof}

    Combined with \eqref{eq:anarchy}, this yields isomorphisms
    \[\Ccoh_T(X; \cF)\per\inv \simeq \Ccoh_T(X^0; i^!\cF)\per\inv \simeq \Ccoh_T(X^-; \iota^{-,!}\cF)\per\inv.\]
    For a torus acting on an affine complex variety and $\bC$-coefficients, this recovers the first isomorphism of \cite[(5.3.3)]{NakajimaLectures}, where $\oH^*_T(\pt)$ is localized at all its nonzero elements.

  \ssec{Donaldson--Thomas perverse sheaves}
  \label{ssec:dt}

    Let $X$ be a $(-1)$-shifted symplectic derived algebraic space, quasi-separated and of finite type over $k = \bC$, equipped with an orientation and a $T$-action preserving the $(-1)$-shifted symplectic structure (see \cite[\S 6.2.2]{Descombes}).
    We then have the Donaldson--Thomas sheaf $\phi_X \in \D^T(X)$ (see \cite[Thm.~4.8]{BBBBJ} and \cite[\S 4.2]{KPS})\footnote{%
      More precisely, the $(-1)$-shifted symplectic structure on $X$ gives rise to a d-critical structure on the classical truncation $X_\cl$.
      The latter is $T$-invariant and descends to $[X_\cl/T]$ (see Lemma~6.9 and the paragraph before (6.20) in \cite{Descombes}).
      An orientation of $X$ is, by definition, an orientation of the d-critical stack $[X_\cl/T]$.
    }, where $\D$ is the Betti weave with $\Lambda = \Q$ coefficients \sssecref{sssec:mashie}.

    Let $\lambda : \Gm \to T$ be a cocharacter.
    We use the notation preceding \corref{cor:hyploc}, so that $X^0$ is the fixed locus of $\lambda$ and $X^-$ is the repeller, i.e., the attractor of $\lambda^{-1}$.
    By \cite[Thm.~5.7]{KPS}, $[X^0/T]$ is the component indexed by $\lambda^{-1}$ of the stack $\on{Grad}([X_\cl/T])$ of \cite[\S 5.1]{KPS}.
    For each connected component $X^0_\alpha$ of $X^0$, let $\phi_{X^0_\alpha} \in \D^T(X^0_\alpha)$ be the Donaldson--Thomas sheaf of $[X^0_\alpha/T]$\footnote{%
      with the d-critical structure and orientation induced from $[X_\cl/T]$, see \cite[Prop.~5.31, \S 6.5]{KPS}
    }, and $I_\alpha \in \Z$ the value on $[X^0_\alpha/T]$ of the index function of \cite[Def.~6.6]{KPS}.

    We have the following concentration result for Donaldson--Thomas sheaves, and hence for critical cohomology:

    \begin{cor}\label{cor:dt}
      There are isomorphisms
      \begin{thmlist}
        \item\label{item:dt/sheaf}
        $i^{!*}\phi_X \simeq \bigoplus_\alpha L_\Sigma(\phi_{X^0_\alpha})[I_\alpha]$ in $\D^T(X^0)\inv$;
        \item\label{item:dt/coh}
        $\Ccoh_T(X; \phi_X)\per\inv \simeq \bigoplus_\alpha \Ccoh_T(X^0_\alpha; \phi_{X^0_\alpha})\per\inv[I_\alpha]$.
      \end{thmlist}
    \end{cor}
    \begin{proof}
      By \cite[Thm.~5.7]{KPS}, the corresponding component of the correspondence $\on{Grad}([X_\cl/T]) \leftarrow \on{Filt}([X_\cl/T]) \to [X_\cl/T]$ is
      \[ [X^0/T] \xleftarrow{\pi^-} [X^-/T] \xrightarrow{\iota^-} [X_\cl/T], \]
      so \cite[Thm.~7.16]{KPS} gives an isomorphism
      \begin{equation}\label{eq:hyp phi}
       	\pi^-_*\iota^{-,!}\phi_X \simeq \bigoplus_\alpha \phi_{X^0_\alpha}[I_\alpha].
      \end{equation}
      By \cite[Rem.~1.7(iii)]{RicharzHyperbolic}, the $\Gm$-action on $X_\cl$ through $\lambda$ is étale-locally linearizable, so \corref{cor:hyploc} gives $i^{!*}\phi_X \simeq L_\Sigma\pi^-_*\iota^{-,!}\phi_X$, whence \itemref{item:dt/sheaf}.

      For \itemref{item:dt/coh}, the isomorphism following \corref{cor:hyploc} gives
      \[ \Ccoh_T(X; \phi_X)\per\inv \simeq \Ccoh_T(X^-; \iota^{-,!}\phi_X)\per\inv \simeq \Ccoh_T(X^0; \pi^-_*\iota^{-,!}\phi_X)\per\inv, \]
      and we conclude by \eqref{eq:hyp phi} above.
    \end{proof}

    The hyperbolic localizations $\pi^+_!\iota^{+,*}\phi_X$ and $\pi^-_*\iota^{-,!}\phi_X$ are computed in \cite[Thm.~6.14]{Descombes} and \cite[Thm.~7.16]{KPS}, respectively; by \corref{cor:hyploc}, both become $i^{!*}\phi_X$ after $\Sigma$-localization.
    Compare \cite[Rem.~8.3.5]{RSYZ}, where the analogous isomorphism after localization at the fraction field of $\oH^*_T(\pt)$ is obtained for moduli of spiked instantons via \cite[Thm.~6.2(2)]{GoreskyKottwitzMacPherson}, and \cite[\S 2.2.5]{COZZ} for global critical loci, where it is shown that the restriction to the fixed locus becomes surjective after such localization.

    \begin{exam}\label{exam:hilb}
      Let $T = \{t_1t_2t_3 = 1\} \sub \bG_m^3$, let $U$ be the noncommutative Hilbert scheme of $n$ points in $\A^3$, i.e., the quotient of $\{(B_1, B_2, B_3, v) : v \text{ cyclic}\}$ by $\GL_n$, with $T$ acting by $B_i \mapsto t_iB_i$, and let $f := \on{Tr}(B_3[B_1, B_2]) : U \to \A^1$.
      Then $f$ is $T$-invariant and its derived critical locus $X := \on{Crit}(f)$ has classical truncation $\on{Hilb}^n(\A^3)$, with $X^T$ the finite set of monomial ideals of colength $n$ (which are in bijection with plane partitions of $n$; see \cite[Lem.~4.1(a) and \S 4]{BehrendFantechi}).
      In this case, we claim:
      \begin{defnlist}
       	\item
        Write $\D^T_{\mrm{ctf}}(X)$ for the full subcategory of constructible objects of finite Tor-dimension (defined on underlying objects of $\D(X)$, as in \defnref{defn:ctf muell}), and $\D^T_{\mrm{ctf}}(X)\inv$ for the idempotent completion of the essential image of $\D^T_\mrm{ctf}(X)$ by $L_\Sigma : \D^T(X) \to \D^T(X)\inv$.
        Then by \corref{cor:equiconcat diag}, the functor $i^{!*}$ determines an equivalence
        \[
          \D^T_\mrm{ctf}(X)\inv \xrightarrow{i^{!*}}
          \prod_{x \in X^T} \D_\mrm{ctf}(BT)\inv \simeq
          \prod_{x \in X^T} \on{Perf}(R_T),
        \]
        where $R_T := \on{End}_{\D(BT)\inv}(L_\Sigma \Q_{BT})$.
        For the rightmost identification, use the fact that, since $BT$ has simply connected homotopy type, $\D_{\mrm{ctf}}(BT)$ is generated as a thick subcategory by the unit.

        \item\label{item:hilb/phi}
        For the canonical orientation on $X$ (see \cite[Ex.~3.20]{KPS}), \corref{cor:dt} yields
        \begin{equation}\label{eq:hilb}
         	i^{!*}\phi_X \simeq \bigoplus_{x \in X^T} L_\Sigma(\Q_{BT})[d_x]
        \end{equation}
        for some integers $d_x$.
        In particular, $\Ccoh_T(X; \phi_X)\per\inv$ is free as a module over $\Ccoh_T(\pt)\per\inv$, of rank the number of plane partitions of $n$.
        Compare \cite[Cors.~6.6, 6.7]{DavisonIntegrality} and \cite[Rem.~10.10]{COZZII}.
      \end{defnlist}

      To see \itemref{item:hilb/phi}, note that $T$ acts with nonzero weights on the Zariski tangent spaces of $X_\cl$ at the finitely many points of $X^T$ (see \cite[Lem.~4.1(b)]{BehrendFantechi}), and apply \corref{cor:dt} with a cocharacter $\lambda : \Gm \to T$ such that $\langle w, \lambda \rangle \ne 0$ for each of these weights $w$.
      One can show (e.g. using properness of the punctual Hilbert scheme) that the $\lambda$-fixed locus $X^0 := (X_\cl)^{\lambda}$ is equal to $X^T$, so that each $\phi_{X^0_\alpha}$ is $\Q_{BT}$ up to shift.
    \end{exam}

    By classical localization theory, e.g. \cite[Thm.~6.2]{GoreskyKottwitzMacPherson}, $\Ccoh_T(X; \phi_X)\per\inv$ reduces to the localized $T$-equivariant cohomology of the stalks of $\phi_X$ at the points of $X^T$.
    However, these stalks are not of rank one in general (e.g. for $n = 4$, the stalk at the ideal $(x, y, z)^2$ has rank three; see \cite[Props.~2.4.2(ii), 3.3.1(ii)]{DimcaSzendroi}).
    The above shows that they become of rank one after $\Sigma$-localization.

\section{Smith theory for \texorpdfstring{$\mu_\ell^r$}{mu_l^r}-actions}
\label{sec:smith}

  \ssec{Smith restriction}
  \label{ssec:smith}

    Let $\ell$ be a prime and $r \ge 1$, and take the diagonalizable group scheme $G$ of \ssecref{ssec:diag} to be $\mu_\ell^r$, so that $\Sigma$ is the set of line bundles associated with the nontrivial characters of $\mu_\ell^r$.
    Let $X$ be a quasi-compact algebraic space with $\mu_\ell^r$-action whose fixed locus $X^{\mu_\ell^r} \to X$ is a closed immersion, and denote by $i : X^{\mu_\ell^r} \to X$ the inclusion.
    By \corref{cor:equiconcat diag} and \remref{rem:equiconcat subgroup}, $i_*$ induces an equivalence $\D^{\mu_\ell^r}(X^{\mu_\ell^r})\inv \simeq \D^{\mu_\ell^r}(X)\inv$, and we have the functor of \eqref{eq:shriekstar},
    \begin{equation}\label{eq:gen smith res}
     	i^{!*} : \D^{\mu_\ell^r}(X) \to \D^{\mu_\ell^r}(X^{\mu_\ell^r})\inv.
    \end{equation}
    This may be regarded as a \emph{Smith restriction} functor for arbitrary topological weaves $\D$.
    On cohomology, it recovers Quillen's localization theorem for elementary abelian $\ell$-groups (see \ssecref{ssec:quillen}).

    For $r = 1$ and $\D$ as in \varntref{varnt:acyclic ell}, we can extend this to stacks, with the fixed locus $X^{\mu_\ell}$ replaced by the image of the homotopy fixed point stack $X^{h\mu_\ell}$ (see \cite[Def.~A.14]{Virloc}).

    \begin{cor}\label{cor:smith hfix}
      Let $\D$ be as in \varntref{varnt:acyclic ell}, and let $\sX \in \Stk$ be quasi-compact with separated diagonal and affine stabilizers.
      Suppose $\mu_\ell$ acts on $\sX$ and $\sZ \to \sX$ is a $\mu_\ell$-equivariant closed immersion whose image contains that of $\sX^{h\mu_\ell} \to \sX$.
      Denote by $i : \sZ \to \sX$ the inclusion.
      Then the functor $i_* : \D^{\mu_\ell}(\sZ)\inv \to \D^{\mu_\ell}(\sX)\inv$ is an equivalence, and for every $\cF \in \D^{\mu_\ell}(\sX)$ the map $i^* : \Ccoh_{\mu_\ell}(\sX; \cF)\per\inv \to \Ccoh_{\mu_\ell}(\sZ; i^*\cF)\per\inv$ is invertible.
    \end{cor}
    \begin{proof}
      Let $x$ be a geometric point of $\sX \setminus \sZ$, and $\alpha : \uAut_{[\sX/\mu_\ell]}(x) \to \mu_{\ell,k(x)}$ as in \thmref{thm:equiconcat}.
      For every closed subgroup $H \sub \uAut_{[\sX/\mu_\ell]}(x)$ with $H \simeq \mu_{\ell,k(x)}$, the restriction $\alpha|_H$ is either trivial or an isomorphism (since $\mu_{\ell,k(x)} \simeq \Z/\ell$ as $k(x)$ is algebraically closed).
      Since $x$ does not lift to $\sX^{h\mu_\ell}$, $\alpha$ does not admit a group-theoretic section (see \cite[Rem.~A.17]{Virloc}), hence we must have $\alpha|_H$ trivial.
      In other words, the line bundle associated with any nontrivial character of $\mu_\ell$ satisfies \itemref{cond:Lell} on $[(\sX \setminus \sZ)/\mu_\ell]$, so we conclude by \varntref{varnt:acyclic ell}.
    \end{proof}

    For example, let $\sX = B\mu_\ell$ with the $\mu_\ell$-action for which $[\sX/\mu_\ell] \simeq B\mu_{\ell^2}$.
    Then $\alpha$ is the surjection $\mu_{\ell^2} \to \mu_\ell$, which has no section, so $\sX^{h\mu_\ell} = \varnothing$ and thus $\D^{\mu_\ell}(\sX)\inv = 0$.
    However, $\St^{\mu_\ell}_\sX(x) = \mu_\ell$ so this cannot be deduced from \thmref{thm:equiconcat}.

  \ssec{Quillen's localization theorem}\label{ssec:quillen}

    In the same setting, \corref{cor:concoh} gives for every $\cF \in \D^{\mu_\ell^r}(X)$ an isomorphism
    \[
      i^* : \Ccoh_{\mu_\ell^r}(X; \cF)\per\inv \to \Ccoh_{\mu_\ell^r}(X^{\mu_\ell^r}; i^*\cF)\per\inv,
    \]
    and similarly with $i^!$.

    Now suppose $k = \bC$, resp. $k$ is algebraically closed with $\ell \in k^\times$, so that $\mu_\ell^r \simeq (\Z/\ell)^r$ is an elementary abelian $\ell$-group, and let $\D$ be the Betti weave, resp. the étale weave with $\Z/\ell$-coefficients, and $\cF$ the constant sheaf on $\bF_\ell$.
    Fixing a primitive $\ell$-th root of unity identifies the characters of $\mu_\ell^r$ with $\Hom(\mu_\ell^r, \bF_\ell) \simeq \oH^1(B\mu_\ell^r; \bF_\ell)$, and $c_1(\cL_\chi)$ then acts as the Bockstein $\beta\chi \in \oH^2(B\mu_\ell^r; \bF_\ell)$, up to sign.
    On cohomology, we thus get an isomorphism
    \[
      i^* : \oH^*_{\mu_\ell^r}(X; \bF_\ell)[e^{-1}] \to \oH^*_{\mu_\ell^r}(X^{\mu_\ell^r}; \bF_\ell)[e^{-1}], \quad e := \textstyle\prod_{\chi \ne 0} \beta\chi.
    \]
    For $X$ separated of finite type, this is \cite[Thm.~4.2]{QuillenSpectrum}\footnote{%
      To apply \cite[Thm.~4.2]{QuillenSpectrum}, one needs that $X(\bC)$ is paracompact of finite cohomological dimension, e.g. $X$ is separated of finite type.
    }, resp. \cite[Cor.~7.4]{IllusieZheng} (which also allows $\ell = \on{char}(k)$).

    \begin{rem}\label{rem:tate}
      For $r = 1$, the localized groups $\oH^*_{\mu_\ell}(X; \bF_\ell)[e^{-1}]$ above are \emph{Tate cohomology} groups.
      Indeed, regarding an object of $\D(B\mu_\ell)$ as a complex of $\bF_\ell[\Z/\ell]$-modules $K$, the class $\bF_\ell \to \bF_\ell[2]$ of \cite[(3.5)]{FengSmith} is a nonzero multiple of $\beta\chi$, so \lemref{lem:homloc} identifies the Tate cohomology $T^i(K)$ of \cite[3.4.1]{FengSmith} with
      \[ \pi_0\Maps_{\D(B\mu_\ell)\inv}(L_\Sigma\un, L_\Sigma K[i]). \]

      For $X$ a separated algebraic space of finite type with $\mu_\ell$-action, and $i : X^{\mu_\ell} \to X$ the inclusion of the fixed locus, we have
      \begin{equation}\label{eq:Feng}
        L_\Sigma f_!\cF \simeq L_\Sigma (f \circ i)_! i^*\cF,
        \qquad \cF \in \D^{\mu_\ell}(X),
      \end{equation}
      where $f : [X/\mu_\ell] \to B\mu_\ell$ is the projection.
      Indeed, $\D^{\mu_\ell}(X \setminus X^{\mu_\ell})\inv = 0$ (\corref{cor:equiconcat diag}, \remref{rem:equiconcat subgroup}), so this follows by applying $f^\Sigma_! \circ L_\Sigma$ to the localization triangle.
      On Tate cohomology, \eqref{eq:Feng} recovers the isomorphism
      \[ T^*\RGamma_c(X; \cF) \simeq \RGamma_c(X^{\mu_\ell}; \mathrm{Psm}(\cF)) \]
      of \cite[3.4.5]{FengSmith}, stated there for $\cF$ bounded and $X$ a separated scheme locally of finite type.
    \end{rem}

  \ssec{Classical Smith theory}

    For $r = 1$, we now explain how $\Sigma$-localization theory recovers classical Smith theory for sheaves as in \cite{TreumannSmith,RicheWilliamsonSmith}.

    For this, let $\Lambda$ be a commutative ring in which every integer prime to $\ell$ is invertible, and take $\D$ to be either the Betti weave with $\Lambda$-coefficients (over $k = \bC$), or, if $\Lambda$ is a noetherian $\Z/\ell^m$-algebra for some $m$, the étale weave with $\Lambda$-coefficients (over $k$ an algebraically closed field with $\ell \in k^\times$).

    \begin{defn}\label{defn:ctf}
      For an algebraic space $X$, we say that $\cF \in \D(X)$ is \emph{constructible of finite Tor-dimension} if there is a finite stratification of $X$ by locally closed subspaces on each of which $\cF$ is isomorphic to a constant perfect complex of $\Lambda$-modules locally in the analytic, resp. étale, topology.
    \end{defn}

    When $\Lambda$ is noetherian, and of finite global dimension in the Betti case, these are the bounded complexes with constructible cohomology sheaves and finite Tor-dimension.\footnote{%
      In the étale case, this is \cite[Rem.~7.2]{HemoRicharzScholbach}.
      In the Betti case, constructibility is in the sense of \cite[Def.~4.1.1]{DimcaSheaves}, and finite Tor-dimension is automatic as $\Lambda$ has finite global dimension.
      Any $\cF$ constructible of finite Tor-dimension has constructible cohomology sheaves, since a perfect complex of $\Lambda$-modules has bounded and finitely generated cohomology.
      Conversely, let $\cF$ be a bounded complex whose cohomology sheaves $\cH^n(\cF)$ are constructible, i.e., locally constant with finitely generated stalks on the strata of some stratification.
      Since $\cF$ is bounded, every point $y$ of a stratum $Y$ has a contractible open neighbourhood $U \sub Y(\bC)$ on which every $\cH^n(\cF)$ is constant.
      Since $\Maps_{\D(U)}(M_U, N_U) \simeq \Maps_{\D(\Lambda)}(M, N)$ for $M, N \in \D(\Lambda)$, induction on the cohomological amplitude using the triangles $\tau_{<n}\cF \to \cF \to \cH^n(\cF)[-n]$ shows that $\cF|_U \simeq P_U$ for some $P \in \D(\Lambda)$ with bounded and finitely generated cohomology, which is perfect as $\Lambda$ is noetherian of finite global dimension.
    }

    \begin{defn}\label{defn:ctf muell}
      For $X$ an algebraic space with $\mu_\ell$-action, we say that $\cF \in \D^{\mu_\ell}(X)$ is \emph{constructible of finite Tor-dimension} if its $*$-inverse image along $X \twoheadrightarrow [X/\mu_\ell]$ is.
      We say that $\cF$ has \emph{perfect stalks} if for every geometric point $x$ of $X$, the $*$-inverse image $\cF_x$ of $\cF$ along $B\St_x \to [X/\mu_\ell]$ is compact as an object of $\D(B\St_x) \simeq \D(\Lambda[\St_x])$, i.e., a perfect complex of $\Lambda[\St_x]$-modules, where $\St_x := \St^{\mu_\ell}_X(x) \sub \mu_\ell$ is the stabilizer (see Footnote~\ref{fn:equiconcat}).
    \end{defn}
    For example, if $\mu_\ell$ acts trivially on $X$, then $[X/\mu_\ell] \simeq X \times B\mu_\ell$, and $\cF \in \D^{\mu_\ell}(X)$ has perfect stalks if and only if $\cF_x \in \D(\Lambda[\mu_\ell])$ is perfect for every geometric point $x$ of $X$.
    We denote by
    \[\D^{\mu_\ell}_\mrm{ctf}(X)_{\mu_\ell\text{-}\mrm{perf}} \sub \D^{\mu_\ell}_\mrm{ctf}(X) \sub \D^{\mu_\ell}(X)\]
    the full subcategories of objects constructible of finite Tor-dimension with perfect stalks, resp. constructible of finite Tor-dimension.

    \begin{thm}[Classical Smith theory]\label{thm:smith}
      Let $X$ be a quasi-compact quasi-separated algebraic space with $\mu_\ell$-action.
      \begin{thmlist}
        \item\label{item:smith/triv}
        Suppose $\mu_\ell$ acts trivially and $\cF \in \D^{\mu_\ell}_\mrm{ctf}(X)$.
        Then $L_\Sigma(\cF) = 0$ if and only if $\cF$ has perfect stalks.

        \item\label{item:smith/ff}
        Suppose $\mu_\ell$ acts trivially, and consider the Verdier quotient $\on{Sm}(X) := \D^{\mu_\ell}_\mrm{ctf}(X) / \D^{\mu_\ell}_\mrm{ctf}(X)_{\mu_\ell\text{-}\mrm{perf}}$.
        Then $\on{Sm}(X)$ is $2$-periodic, i.e., $[2] \simeq \id$, and $L_\Sigma$ induces a fully faithful functor
        \[ \on{Sm}(X) \hookrightarrow \D^{\mu_\ell}(X)\inv. \]

        \item\label{item:smith/restr}
        Suppose $X^{\mu_\ell} \to X$ is a closed immersion, and $\Lambda$ is noetherian and, in the Betti case, of finite global dimension.
        Then for every $\cF \in \D^{\mu_\ell}_\mrm{ctf}(X)$, the objects $i^*\cF$ and $i^!\cF$ lie in $\D^{\mu_\ell}_\mrm{ctf}(X^{\mu_\ell})$, and the cofibre of the canonical morphism $i^!\cF \to i^*\cF$ has perfect stalks.
      \end{thmlist}
    \end{thm}

    For $\Lambda$ a field of characteristic $\ell$, $\on{Sm}(X)$ is the category of sheaves with Tate coefficients of \cite[\S 4.2]{TreumannSmith}, called the Smith category in \cite[\S 1.5]{RicheWilliamsonSmith}.
    Thus \itemref{item:smith/ff} realizes the Smith category as the constructible part of $\D^{\mu_\ell}(X)\inv$, and shows that it is $2$-periodic.
    (For $X = \pt$, $2$-periodicity is shown in \cite[Prop.~4.2]{TreumannSmith}, cf. \cite[Lem.~3.9]{RicheWilliamsonSmith}.)

    Part~\itemref{item:smith/restr} recovers \cite[Thm.~4.1]{TreumannSmith} for complex algebraic varieties and its étale analogue \cite[Lem.~3.5]{RicheWilliamsonSmith}.
    It says that the two functors $i^*$ and $i^!$ agree after composition with $\D^{\mu_\ell}_\mrm{ctf}(X^{\mu_\ell}) \to \on{Sm}(X^{\mu_\ell})$; the resulting functor is the ``Smith restriction'' functor of \cite[Def.~4.2, Rem.~4.3]{TreumannSmith},
    \[
      i^{!*} : \D^{\mu_\ell}_\mrm{ctf}(X) \to \on{Sm}(X^{\mu_\ell}),
    \]
    and by \itemref{item:smith/ff}, it is the restriction of $i^{!*} : \D^{\mu_\ell}(X) \to \D^{\mu_\ell}(X^{\mu_\ell})\inv$ \eqref{eq:gen smith res} to $\D^{\mu_\ell}_\mrm{ctf}(X)$.

    \begin{lem}\label{lem:ctf}
      Let $X$ be a qcqs algebraic space with $\mu_\ell$-action.
      \begin{thmlist}
        \item\label{item:ctf/thick}
        $\D^{\mu_\ell}_\mrm{ctf}(X) \sub \D^{\mu_\ell}(X)$ is a thick subcategory closed under the twists $\vb{n}$, $n\in\Z$.

        \item\label{item:ctf/restr}
        Suppose $X^{\mu_\ell} \to X$ is a closed immersion, and $\Lambda$ is noetherian and, in the Betti case, of finite global dimension.
        Then $i^*$ and $i^!$ send $\D^{\mu_\ell}_\mrm{ctf}(X)$ into $\D^{\mu_\ell}_\mrm{ctf}(X^{\mu_\ell})$.
      \end{thmlist}
    \end{lem}
    \begin{proof}
      Both claims may be checked after $*$-inverse image along the étale surjection $X \twoheadrightarrow [X/\mu_\ell]$, which commutes with $\vb{1}$, $i^*$ and $i^!$.
      We may also assume that $X$ is a scheme.

      For \itemref{item:ctf/thick}, it suffices to show that for every locally closed $Y \sub X$, the objects of $\D(Y)$ locally isomorphic to a constant perfect complex form a thick subcategory.
      In the étale case, this follows from the fact that these are exactly the dualizable objects (see \cite[\S 7]{HemoRicharzScholbach}).
      In the Betti case, let $f : \cF \to \cG$ be a morphism between such objects.
      Every point of $Y(\bC)$ has a contractible open neighbourhood $U$ on which $\cF \simeq M_U$ and $\cG \simeq N_U$ are both constant; since $\Maps_{\D(U)}(M_U, N_U) \simeq \Maps_{\D(\Lambda)}(M, N)$, $f|_U$ is constant and hence so is $\Cofib(f)|_U$.
      For retracts one argues similarly.
      Closure under $\vb{n}$ is clear since $\un\vb{1}$ is $\otimes$-invertible and dualizable.

      For \itemref{item:ctf/restr}, since $\Lambda$ is noetherian, the objects of $\D_{\mrm{ctf}}$ are the bounded complexes with constructible cohomology sheaves and finite Tor-dimension, so it suffices to recall that these properties are preserved by $i^*$ and $i^!$; see e.g. \cite[Thm.~4.1.5]{DimcaSheaves} for the Betti case, and \cite[Exp.~XVII, Thm.~5.2.11]{SGA4} and \cite[Props.~6.4.4(1), 6.4.5]{LiuZheng} for the \'etale case (where the statement for $i^!$ follows from those for $i^*$ and $j_*$ using the localization triangle $i_*i^! \to \id \to j_*j^*$).
    \end{proof}

    We first consider the case of a point.

    \begin{lem}\label{lem:smithpt}
      Let $\cF \in \D(B\mu_\ell)$.
      If $\cF$ is dualizable, then $L_\Sigma(\cF) = 0$ if and only if $\cF$ is compact.
    \end{lem}
    \begin{proof}
      For a nontrivial character $\chi$ of $\mu_\ell$, let $q_\chi : \V(\cL_\chi) \setminus 0 \to B\mu_\ell$ be the complement of the zero section of the associated line bundle.
      By the localization and homotopy invariance properties of $\D$, one obtains the exact triangle
      \[ \un \xrightarrow{c_1(\cL_\chi)} \un\vb{1} \to q_{\chi,*}\un\vb{1}. \]
      Since $\V(\cL_\chi) \setminus 0 \simeq [\Gm/\mu_\ell]$, where $\mu_\ell$ acts freely on $\Gm$ via $\chi$, the object $q_{\chi,*}\un$ is compact (in the Betti case by the proof of \cite[Lem.~4.1]{TreumannSmith}, which applies to any $\Lambda$; in the étale case by the proof of \cite[Prop.~2.5]{RicheWilliamsonSmith}, which applies to any noetherian $\Lambda$ killed by a power of $\ell$).
      In particular, the cofibre of $c_1(\cL_\chi) : \un \to \un\vb{1}$ is compact.
      Then since any $s \in S$ is a composite of twists of such morphisms $c_1(\cL_\chi)$, it follows by the octahedral axiom that $\Cofib(s)$ is compact.

      Now suppose $L_\Sigma(\cF) = 0$.
      By \lemref{lem:homloc} there exists an $s : \un \to \un\vb{n}$ in $S$ with $s \otimes \id_\cF : \cF \to \cF\vb{n}$ null-homotopic, whence the splitting $\Cofib(s) \otimes \cF \simeq \cF\vb{n} \oplus \cF[1]$.
      Since $\cF$ is dualizable, this object is also compact and hence so is $\cF$.

      Conversely, suppose $\cF$ is compact.
      Write $\cR := p_*\un$ for $p : \pt \twoheadrightarrow B\mu_\ell$ the quotient map.
      Since $p^* \simeq p^!$ is conservative and colimit-preserving, $\cR$ is a compact generator of $\D(B\mu_\ell)$; equivalently, the thick subcategory of compact objects of $\D(B\mu_\ell)$ is generated by $\cR$.
      Thus in order to show that $L_\Sigma(\cF) = 0$ for every compact $\cF$, it will suffice to show that $L_\Sigma(\cR) = 0$, since the objects killed by $L_\Sigma$ form a thick subcategory.
      For this, observe that by the smooth base change formula, we have
      \[\pi_0\Maps(\cR, \cR\vb{1}) \simeq \pi_0\Maps(p^*p_*\un, \un\vb{1}) \simeq \bigoplus_{\mu_\ell} \pi_0\Maps_{\D(\pt)}(\un, \un\vb{1}) \simeq 0,\]
      hence $L_\Sigma(\cR) = 0$ by \lemref{lem:nilp}.
    \end{proof}

    \begin{lem}\label{lem:ff}
      Given any thick subcategory $\cC \sub \D(X)$ closed under $\vb{1}$, let $\cC_0 \sub \cC$ be the full subcategory of $\cF$ with $L_\Sigma(\cF) = 0$.
      Then $L_\Sigma$ induces a fully faithful functor $\cC/\cC_0 \to \D(X)\inv$.
    \end{lem}
    \begin{proof}
      Let $\cK \sub \D(X)$ be the full subcategory of objects killed by $L_\Sigma$, so that $\D(X)\inv \simeq \D(X)/\cK$ and $\cC_0 = \cC \cap \cK$.
      By \cite[Lem.~4.7.1]{KrauseLocalization} (applied to homotopy categories via \cite[Thm.~I.3.3]{NikolausScholze}), it suffices to show that every morphism $g : \cG \to \cF$ with $\cG \in \cK$ and $\cF \in \cC$ factors through some object of $\cC_0$.
      By \lemref{lem:homloc}, there exists an $s : \un_X \to \un_X\vb{n}$ in $S$ such that $s \otimes \id_\cG : \cG \to \cG\vb{n}$ is null.
      In the commutative square
      \[\begin{tikzcd}
        \cG \ar{r}{g}\ar[swap]{d}{s \otimes \id_\cG}
        & \cF \ar{d}{s \otimes \id_\cF}
        \\
        \cG\vb{n} \ar{r}{g\vb{n}}
        & \cF\vb{n}
      \end{tikzcd}\]
      the left-hand vertical arrow is null-homotopic, so $g$ factors through $\cH := \Fib(s \otimes \id_\cF)$.
      The object $\cH$ lies in $\cC$, since $s \otimes \id_\cF$ is a composite of morphisms between objects $\cF\vb{m}$ lying in $\cC$; it also lies in $\cK$, since $L_\Sigma(s \otimes \id_\cF)$ is invertible.
    \end{proof}

    \begin{proof}[Proof of \thmref{thm:smith}]\leavevmode

      \itemref{item:smith/triv}:
      Choose a finite stratification of $X$ by locally closed subspaces $Y$ such that each $\cF|_Y$ is locally isomorphic to a constant perfect complex.
      If $L_\Sigma(\cF) = 0$, then for every geometric point $x$ of $X$, the $*$-inverse image $\cF_x$ along $B\mu_\ell \to [X/\mu_\ell]$ satisfies $L_\Sigma(\cF_x) = 0$, and $\cF_x$ is dualizable since $\cF$ is constructible of finite Tor-dimension; hence $\cF_x$ is compact by \lemref{lem:smithpt}.

      Conversely, suppose $\cF$ has perfect stalks.
      By \lemref{lem:strat}, we may assume that $\cF|_X$ is locally isomorphic to a constant perfect complex.
      We claim that $\pi_*(\cE)$ is bounded, where $\cE := \uHom(\cF, \cF)$ and $\pi : X \times B\mu_\ell \to X$ is the projection.
      Since $X$ has finite cohomological dimension, it will then follow that $\pi_0\Maps(\cF, \cF\vb{n}) \simeq \oH^{2n}(X; \pi_*(\cE)(n)) \simeq 0$ for $n \gg 0$, hence $L_\Sigma(\cF) \simeq 0$ by \lemref{lem:nilp}.

      Note that $\cE$ is locally constant with compact stalks $\cE_x \simeq \cF_x^\vee \otimes \cF_x$.
      Moreover, the complexes $R\Gamma(B\mu_\ell, \cF_x^\vee \otimes \cF_x)$ are bounded: for every compact $\cG \in \D(B\mu_\ell)$, $R\Gamma(B\mu_\ell, \cG)$ is bounded, since this holds for the compact generator (see the proof of \lemref{lem:smithpt}).
      Under the equivalence $\D(X \times B\mu_\ell) \simeq \Fun(B\mu_\ell, \D(X))$, $\pi_*(\cE)$ is $\on{RHom}_{\Lambda[\mu_\ell]}(\Lambda, \cE)$; using the $2$-periodic free resolution
      \[
        \cdots \to \Lambda[\mu_\ell] \xrightarrow{N} \Lambda[\mu_\ell] \xrightarrow{1-g} \Lambda[\mu_\ell] \to \Lambda,
      \]
      where $g$ is a generator of $\mu_\ell$ and $N = \sum_i g^i$, this is the totalization of the complex
      \[\cE \xrightarrow{1-g} \cE \xrightarrow{N} \cE \xrightarrow{1-g} \cdots.\]
      Since $\cE$ is bounded, each degree of this totalization involves only finitely many terms.
      It follows that the formation of $\pi_*(\cE)$ commutes with stalks, i.e., the stalk at $x$ is the bounded complex $R\Gamma(B\mu_\ell, \cE_x)$.
      As $\cE$ is locally constant and $X$ is quasi-compact, this boundedness holds uniformly in $x$, hence $\pi_*(\cE)$ is bounded as claimed.

      \itemref{item:smith/ff}:
      By \lemref{lem:ctf}\itemref{item:ctf/thick}, \lemref{lem:ff} applies to $\cC = \D^{\mu_\ell}_\mrm{ctf}(X)$, and we have $\cC_0 = \D^{\mu_\ell}_\mrm{ctf}(X)_{\mu_\ell\text{-}\mrm{perf}}$ by \itemref{item:smith/triv}.
      For periodicity, note that for any $\cL \in \Sigma$ the morphisms $c_1(\cL) \otimes \id_\cF : \cF \to \cF\vb{1}$ are natural in $\cF$ and inverted by $L_\Sigma$, so $\vb{1} \simeq \id$ on $\D^{\mu_\ell}(X)\inv$.
      Since $\vb{1} \simeq [2]$ (in the étale case, after choosing an isomorphism $\Lambda(1) \simeq \Lambda$, which is possible as $k$ is algebraically closed), $\D^{\mu_\ell}(X)\inv$ is $2$-periodic, hence so is its full subcategory $\on{Sm}(X)$.

      \itemref{item:smith/restr}:
      By \lemref{lem:ctf}\itemref{item:ctf/restr}, the cofibre of $i^!\cF \to i^*\cF$ lies in $\D^{\mu_\ell}_\mrm{ctf}(X^{\mu_\ell})$, is killed by $L_\Sigma$ since $L_\Sigma i^! \simeq L_\Sigma i^*$ \eqref{eq:gen smith res}, and hence has perfect stalks by \itemref{item:smith/triv}.
    \end{proof}

\section{Categorical localization formula}
\label{sec:locfrmlsm}

  In this section we prove a categorification of the Atiyah--Bott--Berline--Vergne localization formula (see \thmref{thm:catlocsm}).
  This is a warm-up for the virtual version, to be proven in \secref{sec:catloc}.

  In this section, and for the rest of the paper, we take the diagonalizable group scheme $G$ to be a torus, and denote it by $T$.

  \ssec{Invertibility of Euler classes}

    Let $Z \to X$ be as in \corref{cor:equiconcat diag}, and suppose that $X$ and $Z$ are \emph{smooth}, or more generally that the inclusion $Z \to X$ is quasi-smooth (i.e., the relative cotangent complex $\cL_{Z/X}$ lies in $\Dperf^{[-1,-1]}(Z)$).
    In this case there is a canonical natural transformation
    \begin{equation}\label{eq:trader}
      \tr_i : i_*i^*\vb{-c} \to \id
    \end{equation}
    called the \emph{trace} (see \cite[4.3.1]{DegliseJinKhan} and \cite[Rem.~3.8]{KhanVirtual}), where $c$ is the rank of $\cL_{Z/X}[-1]$.
    Recall also the Euler transformation
    \begin{equation}\label{eq:proboscidial}
      \eul_{\cN_{\sZ/\sX}} : \id \to \vb{c}
    \end{equation}
    associated with the locally free sheaf $\cN_{\sZ/\sX} = \cL_{\sZ/\sX}[-1]$, which can be described as the composite
    \[
      \id \simeq \pi_!\pi^*\vb{c}
      \xrightarrow{\mrm{unit}} \pi_!0_!0^*\pi^*\vb{c}
      \simeq \vb{c}
    \]
    where $\pi : \V_\sZ(\cN_{\sZ/\sX}) \to \sZ$ is the projection of the normal bundle and the isomorphism comes from homotopy invariance and purity.

    By \eqref{eq:discouragingness}, $\tr_i$ and $\eul_{\cN_{\sZ/\sX}}$ descend to natural transformations
    \[
      \tr_i^\Sigma : i_*^\Sigma i^*_\Sigma \vb{-c} \to \id_{\D^T(X)\inv},
      \qquad
      \eul_{\cN_{\sZ/\sX}}^\Sigma : \id_{\D^T(Z)\inv} \to \vb{c}.
    \]
    When $Z$ has finite stabilizers and trivial $T$-action, the following implies that $\eul_{\cN_{\sZ/\sX}}^\Sigma$ is invertible as long as $\cN_{\sZ/\sX}$ has no fixed part (see \cite[Def.~1.2]{Virloc} for the definitions of the \emph{fixed} and \emph{moving} parts of a complex on $\sZ \simeq Z \times BT$).

    \begin{cor}\label{cor:unretentive}
      Let $X \in \Stk$ be quasi-compact and let $\cE$ be a connective coherent complex on $\sX := X\times BT$, i.e., a coherent complex on $X$ that is equivariant with respect to the trivial $T$-action.
      If $X$ has finite stabilizers and $\cE$ has no fixed part, then the zero section of the derived cone $\sE := \V_\sX(\cE) \to \sX$ induces a canonical equivalence
      \[
        0_* : \D^T(X)\inv \to \D(\sE)\inv.
      \]
    \end{cor}
    \begin{proof}
      By derived invariance, we may assume that $\cE$ is $0$-truncated (i.e., a coherent sheaf).
      Then \cite[Prop.~4.8]{constack} shows that the conditions of \corref{cor:equiconcat diag} are satisfied under our assumptions.
    \end{proof}

  \ssec{Categorical localization formula}

    When $i$ is quasi-smooth, the ``categorical localization formula'' computes the inverses of the isomorphism $\id \to i_*^\Sigma i^*_\Sigma$ in terms of the trace:

    \begin{thm}[Categorical localization formula]\label{thm:catlocsm}
      Let $X \in \Stk$ be quasi-compact with $T$-action, and $Z \to X$ a $T$-equivariant closed immersion such that $\St^T_X(x) \subsetneq T_x$ for every geometric point $x \in X \setminus Z$, and such that the $T$-action on $Z$ is trivial.
      Denote by $i : \sZ \to \sX$ the induced morphism between $\sZ := [Z/T] \simeq Z \times BT$ and $\sX := [X/T]$.
      If $Z \to X$ is quasi-smooth, $Z$ has finite stabilizers, and $\cN_{\sZ/\sX}$ has no fixed part, then the natural transformation of endofunctors of $\D^T(X)\inv$
      \[
        i_*^\Sigma i^*_\Sigma
        \xrightarrow{\eul^{-1}} i_*^\Sigma i^*_\Sigma \vb{-c}
        \xrightarrow{\tr^\Sigma_i} \id,
      \]
      where $\eul^{-1}$ is an inverse of $\eul_{\cN_{\sZ/\sX}}^\Sigma$, is an inverse to the unit $\id \to i_*^\Sigma i^*_\Sigma$.
    \end{thm}

    In view of the invertibility of $\eul_{\cN_{Z/X}}$ on $\D^T(Z)\inv$, \thmref{thm:catlocsm} is a consequence of the following categorification of the self-intersection formula:

    \begin{prop}[Categorical self-intersection formula]\label{prop:ethnogeographically}
      Let $i : \sZ \to \sX$ be a quasi-smooth closed immersion in $\Stk$.
      Then the composite natural transformation
      \[
        i_* i^* \vb{-c}
        \xrightarrow{\tr_i} \id
        \xrightarrow{\mrm{unit}} i_* i^*
      \]
      is canonically identified with $\eul_{\cN_{\sZ/\sX}} : i_* i^* \vb{-c} \to i_* i^*$.
    \end{prop}

    This statement is easily derived from \cite[Cor.~3.17]{KhanVirtual}; see also \corref{cor:antilacrosse} below.

    \begin{cor}\label{cor:buttle}
      For $i$ as in \propref{prop:ethnogeographically}, the composite
      \[
        i^* \vb{-c}
        \xrightarrow{i^* \mrm{unit}} i^*i_*i^*\vb{-c}
        \xrightarrow{i^* \tr_i} i^*
      \]
      is canonically identified with $\eul_{\cN_{\sZ/\sX}} : i^*\vb{-c} \to i^*$.
    \end{cor}
    \begin{proof}
      By the triangle identities, it will suffice to identify the composite
      \[\begin{tikzcd}
        i^*\vb{-c}
        \xrightarrow{i^* \mrm{unit}} i^*i_*i^*\vb{-c}
        \xrightarrow{i^* \tr_i} i^*
        \xrightarrow{i^* \mrm{unit}} i^*i_*i^*
        \xrightarrow{\mrm{counit}(i^*)} i^*
      \end{tikzcd}\]
      since the last two arrows collapse to the identity.
      By \propref{prop:ethnogeographically} this is identified with the total composite in the commutative diagram below:
      \[\begin{tikzcd}
        i^*\vb{-c} \ar{r}{i^* \mrm{unit}}\ar[equals]{rd}
        & i^*i_*i^*\vb{-c} \ar{r}{\eul}\ar{d}{\mrm{counit}(i^*)}
        & i^*i_*i^* \ar{d}{\mrm{counit}(i^*)}
        \\
        & i^*\vb{-c} \ar{r}{\eul}
        & i^*,
      \end{tikzcd}\]
      whence the claim.
    \end{proof}

    \begin{proof}[Proof of \thmref{thm:catlocsm}]
      All operations are considered on $\D(-)\inv$ and we omit the decorations ``$\Sigma$'' from the notation.
      \propref{prop:ethnogeographically} yields an identification between the composite
      \[
        i_*i^*
        \xrightarrow{\eul^{-1}} i_*i^*\vb{-c}
        \xrightarrow{\tr_i} \id
        \xrightarrow{\mrm{unit}} i_*i^*
      \]
      and the identity.
      The other composite
      \[
        \id
        \xrightarrow{\mrm{unit}} i_*i^*
        \xrightarrow{\eul^{-1}} i_*i^* \vb{-c}
        \xrightarrow{\tr_i} \id,
      \]
      is equivalent, after applying $i^*$ on the left, to
      \[
        i^*
        \xrightarrow{\eul^{-1}} i^*\vb{-c}
        \xrightarrow{i^*\mrm{unit}} i^*i_*i^* \vb{-c}
        \xrightarrow{i^*\tr_i} i^*,
      \]
      which is the identity by \corref{cor:buttle}.
      Since $i^*$ is an equivalence (\corref{cor:equiconcat diag}), the claim follows.
    \end{proof}

\section{Virtual Euler and trace}

  Our next goal is to extend the categorical localization formula (\thmref{thm:catlocsm}) to the case of \emph{quasi-}smooth stacks.
  Since the closed immersion $i : \sZ \to \sX$ need no longer be quasi-smooth, we begin in this section by making sense of the trace $\tr_i$ as well as the Euler transformation $\eul_{\cN_{\sZ/\sX}}$ associated with the perfect complex $$\cN_{\sZ/\sX} = \cL_{\sZ/\sX}[-1] \in \Dperf^{[-1,0]}(\sX).$$

  We maintain the notation $T$ and $\Sigma \sub \Pic(BT)$ of the previous section.
  Throughout this section and in the rest of this paper we work under the assumption that all stacks in $\bS$ have affine stabilizers.

  \ssec{Quasi-smooth bundles}

    Given a perfect complex $\cE \in \Dperf(X)$, the associated ``generalized vector bundle'' $\pi : E = \V_X(\cE) \to X$ is the derived stack with functor of points
    \[ (T \xrightarrow{t} X) \mapsto \Maps_{\Dperf(X)}(t^*\cE, \cO_T). \]
    In \secref{sec:locfrmlsm} we made use of the fact that when $\cE$ is of Tor-amplitude $\le 0$ and virtual rank $r$, the projection $\pi$ is smooth so that we have an invertible natural transformation
    \[
      \tr_\pi : \pi_!\pi^*\vb{r} \to \id,
    \]
    see \cite[Prop.~A.10, Thm.~A.13, Constr.~A.16]{KhanVirtual}.

    Now suppose $\cE$ is a perfect complex on $\sX := X\times BT$, i.e., a perfect complex on $X$ which is equivariant with respect to the trivial $T$-action.
    Suppose that $\cE$ is of Tor-amplitude $\le 1$, so that $\pi : \V_\sX(\cE) \to \sX$ is only \emph{quasi-}smooth.
    We still have the trace (see \cite[\S 3.1, Rem.~3.8]{KhanVirtual})
    \begin{equation*}
      \tr_\pi : \pi_!\pi^*\vb{r} \to \id
    \end{equation*}
    but it is typically not invertible.
    However, we will show that this becomes true $\Sigma$-locally, as long as the fixed part $\cE^{\fix}$ is of Tor-amplitude $\le 0$.

    \begin{prop}\label{prop:Htpy-Inv}
      Let $X \in \Stk$ be quasi-compact and $\sX = X \times BT$.
      Let $\cE \in \Dperf^{\ge -1}(\sX)$ of virtual rank $r$ with $\cE^\fix\in\Dperf^{\ge 0}(\sX)$.
      If $X$ has finite stabilizers, then the trace of $\pi : \V_\sX(\cE) \to \sX$ induces an isomorphism
      \[
        \tr_\pi : \pi^\Sigma_!\pi_\Sigma^*\vb{r} \to \id
      \]
      of endofunctors of $\D^T(X)\inv$.
    \end{prop}
    \begin{proof}
      Note that $\pi : \sE := \V_\sX(\cE) \to \sX$ factors through $\pi^\mov : \sE^\mov \to \sX$ and $\sE \to \sE^\mov$, which is a torsor under the vector bundle stack $\pi^\fix : \sE^\fix \to \sX$.
      By homotopy invariance for vector bundle stacks \cite[Prop.~A.10]{KhanVirtual} we may replace $\cE$ by $\cE^\mov$ and assume that $\cE$ has no fixed part.

      Using the localization triangle and stratifying $X$ by global quotient stacks, we may assume that $X$ has the resolution property.
      Arguing as in the proof of \cite[Prop.~A.10]{KhanVirtual} by induction on the Tor-amplitude of $\cE$, we reduce to the case where
      \[
        \cE = \Cofib(\cE^{-1} \to \cE^0) \in \Dperf^{[-1,0]}(\sX)
      \]
      with $\cE^{-1}, \cE^0 \in \Dperf^{[0,0]}(\sX)$.
      We have a commutative diagram
      \begin{equation}\label{eq:HI}
        \begin{tikzcd}
          \sX \ar{r}{0}\ar[swap]{rd}{0_{\sE_0}}
          & \sE \ar{r}{\pi}\ar{d}{s}
          & \sX \ar{d}{0_{\sE_1}}
          \\
          & \sE_0 \ar{r}{p}\ar[swap]{rd}{\pi_{\sE_0}}
          & \sE_1 \ar{d}{\pi_{\sE_1}}
          \\
          & & \sX
        \end{tikzcd}
      \end{equation}
      where $\sE_0 = \V_\sX(\cE^0)$ and $\sE_1 = \V_\sX(\cE^{-1})$ and the square is homotopy cartesian.

      All operations below are considered on $\D(-)\inv$ and we omit the decorations ``$\Sigma$''.
      Since $\cE_1$ has no fixed part and $X$ has finite stabilizers, $0_{\sE_1}^*$ and $0_{\sE_1,!}$ are equivalences (\corref{cor:unretentive}), so it will suffice to show that
      \[
        \tr_\pi : 0_{\sE_1,!}\pi_!\pi^*\vb{r} \to 0_{\sE_1,!}
      \]
      is invertible.
      Under the base change isomorphism, this is identified with (see \cite[Cor.~ 2.5.6]{DegliseJinKhan})
      \[
        \tr_p : p_!p^! 0_{\sE_1,!} \vb{r} \to 0_{\sE_1,!}.
      \]
      Thus it is sufficient to show that $\tr_p : p_!p^*\vb{r} \to \id$ is invertible.

      Since $0_{\sE_1}^*$ and $0_{\sE_1,!}$ are equivalences, the same holds for $\pi_{\sE_1}^*$ and $\pi_{\sE_1,!}$ by functoriality.
      Thus it will suffice to show that
      \[
        \pi_{\sE_0,!}\pi_{\sE_0}^*\vb{r}
        \simeq \pi_{\sE_1,!}p_!p^*\pi_{\sE_1}^*\vb{r}
        \xrightarrow{\tr_p} \pi_{\sE_1,!}\pi_{\sE_1}^*
      \]
      is invertible.
      By functoriality of the trace, this follows from the invertibility of $\tr_{\sE_0}$ and $\tr_{\sE_1}$ (which holds since $\cE_0$ and $\cE_1$ are of Tor-amplitude $[0,0]$, see the discussion above).
    \end{proof}

  \ssec{Virtual Euler and trace}

    We use \propref{prop:Htpy-Inv} to define generalizations of the Euler and trace transformations at the level of localized equivariant derived categories.

    \begin{constr}[Euler transformation]\label{constr:postcarotid}
      Let $X \in \Stk$ be quasi-compact and write $\sX = X\times BT$.
      Let $\cE \in \Dperf^{[-1,0]}(\sX)$ of virtual rank $r$, with $\cE^\fix\in\Dperf^{[0,0]}(\sX)$.
      The \emph{Euler transformation} associated with $\cE$ is the natural transformation
      \[
        \eul_\cE^\Sigma : \id_{\D^T(X)\inv} \to \vb{r}
      \]
      of endofunctors of $\D^T(X)\inv$, defined as follows.
      Denote by $\pi : \V_\sX(\cE) \to \sX$ the projection and by $0 : \sX \to \V_\sX(\cE)$ the zero section.
      Note that the latter is a closed immersion since $\cE \in \Dperf^{\le 0}(\sX)$, so that the unit of the adjunction $(0^*_\Sigma, 0_!^\Sigma)$ gives rise to a canonical natural transformation
      \[
        \id \simeq \pi_!^\Sigma \pi^*_\Sigma \vb{r}
        \xrightarrow{\mrm{unit}} \pi_!^\Sigma 0_!^\Sigma 0^*_\Sigma \pi^*_\Sigma \vb{r}
        \simeq \vb{r}
      \]
      where the first isomorphism is the inverse of the trace (\propref{prop:Htpy-Inv}).
    \end{constr}

    \begin{prop}\label{prop:trophoderm}
      In the situation of \constrref{constr:postcarotid}, suppose moreover that $\cE^\fix \simeq 0$.
      Then the unit $\id \to 0_!^\Sigma 0^*_\Sigma$ is invertible.
      In particular, the Euler transformation $\eul_\cE^\Sigma : \id_{\D^T(X)\inv} \to \vb{r}$ is invertible.
    \end{prop}

    \begin{proof}
      When $\cE$ is of Tor-amplitude $[0,0]$, i.e., $\pi : E := \V_\sX(\cE) \to \sX$ is a vector bundle, the zero section $0: \sX \to E$ satisfies the assumptions of \corref{cor:equiconcat diag} by \cite[Prop.~4.8]{constack} since $\cE^\fix \simeq 0$ and $X$ has finite stabilizers.

      In general, we stratify $X$ by global quotient stacks and use the localization triangle to reduce to the case where $X$ has the resolution property.
      We can then write
      \[ \cE \simeq \Cofib(\cE^{-1} \to \cE^0) \in \Dperf^{[-1,0]}(\sX) \]
      with $\cE^{-1}, \cE^0 \in \Dperf^{[0,0]}(\sX)$.
      We adopt the notation of \eqref{eq:HI}.
      Since $s^*$ and $s_!$ are equivalences by \corref{cor:equiconcat diag} (we omit the decorations ``$\Sigma$'' throughout the proof), it will suffice to show that the lower horizontal arrow below is invertible:
      \[\begin{tikzcd}
        \id \ar{r}{\mrm{unit}}\ar{d}{\mrm{unit}}
        & 0_{E_0,!} 0_{E_0}^* \ar[equals]{d}
        \\
        s_!s^* \ar{r}{\mrm{unit}}
        & s_!0_!0^*s^*
      \end{tikzcd}\]
      But the lower horizontal and right-hand vertical arrows are invertible by \corref{cor:equiconcat diag}, so the claim follows.
    \end{proof}

    \begin{constr}[Trace transformation]\label{constr:undeliberating}
      Let $X \in \Stk$ be quasi-compact with $T$-action, and let $Z \to X$ be a $T$-equivariant closed immersion where the $T$-action on $Z$ is trivial.
      Write $\sZ := [Z/T] \simeq Z \times BT$ and $\sX := [X/T]$.
      Assume that $\cL_{\sZ/\sX} \in \Dperf^{[-2,-1]}(\sZ)$, $\cL_{\sZ/\sX}^\fix \in \Dperf^{[-1,-1]}(\sZ)$, and $Z$ has finite stabilizers.
      Let $i : \sZ \to \sX$ denote the inclusion, $\pi : N_{\sZ/\sX} := \V_\sZ(\cN_{\sZ/\sX}) \to \sZ$ the derived normal bundle, and $c$ the virtual rank of $\cN_{\sZ/\sX} = \cL_{\sZ/\sX}[-1]$.

      Recall that there is a specialization transformation (see \cite[Constr.~3.1]{KhanVirtual}, \cite[Constr.~8.7]{Sixstack})
      \[ \mrm{sp}_i : i_!\pi_!\pi^*i^* \to \id \]
      defined by deformation to the derived normal bundle.
      It gives rise to a natural transformation $\mrm{sp}_i^\Sigma$ on localized functors by \eqref{eq:discouragingness}.
      The \emph{trace transformation}
      \begin{equation}\label{eq:dietotherapy}
        \tr_i^\Sigma : i_*^\Sigma i^*_\Sigma \vb{-c} \to \id
      \end{equation}
      of functors $\D^T(X)\inv \to \D^T(X)\inv$ is defined as the composite
      \[
        i_*^\Sigma i^*_\Sigma \vb{-c}
        \simeq i_*^\Sigma \pi_!^\Sigma \pi^*_\Sigma i^*_\Sigma
        \xrightarrow{\mrm{sp}_i^\Sigma} \id
      \]
      where the isomorphism $\pi_!^\Sigma \pi^*_\Sigma \simeq \vb{-c}$ is the inverse of $\tr_\pi$ (\propref{prop:Htpy-Inv}).
    \end{constr}

    \begin{rem}
      It is immediate from the constructions that if $\cE \in \Dperf^{[-1,0]}(\sX)$ in fact belongs to $\Dperf^{[0,0]}$, then the Euler transformation $\eul_\cE : \id \to \vb{r}$ on $\D^T(X)\inv$ (\constrref{constr:postcarotid}) is the restriction of the usual one on $\D^T(X)$.
      Similarly, if $i : \sZ \to \sX$ is in fact quasi-smooth (i.e., $\cL_{\sZ/\sX} \in \Dperf^{[-2,-1]}(\sZ)$ belongs to $\Dperf^{[-1,-1]}(\sZ)$), then the trace transformations on $\Sigma$-localized categories are the restrictions of the usual ones on unlocalized categories.
    \end{rem}

\section{Categorical virtual localization formula}
\label{sec:catloc}

  We now prove the categorical virtual localization formula, extending \thmref{thm:catlocsm} to the case of quasi-smooth stacks.

  \ssec{The localization formula}

    Let $X\in\Stk$ be quasi-compact with $T$-action, and let $Z \to X$ be a $T$-equivariant closed immersion such that $Z$ contains all $T$-fixed points (that is, for every geometric point $x \in X \setminus Z$ we have $\St^T_X(x) \neq T_x$).
    Denote by $i : \sZ \to \sX$ the induced morphism between $\sZ := [Z/T]$ and $\sX := [X/T]$.
    The categorical concentration theorem (\corref{cor:equiconcat diag}) implies that the unit $\id \to i_*^\Sigma i^*_\Sigma$ and counit $i_*i^! \to \id$ are invertible on $\D^T(X)\inv$.
    The ``categorical localization formula'' computes their inverses in terms of the trace (\constrref{constr:undeliberating}), in the situation where the latter is defined:

    \begin{thm}[Categorical virtual localization formula]\label{thm:catloc}
      Let $X\in\Stk$ be quasi-compact with $T$-action, and let $Z \to X$ be a $T$-equivariant closed immersion  such that $\St^T_X(x) \subsetneq T_x$ for every geometric point $x \in X \setminus Z$, and such that the $T$-action on $Z$ is trivial.
      Denote by $i : \sZ \to \sX$ the induced morphism between $\sZ := [Z/T] \simeq Z \times BT$ and $\sX := [X/T]$.
      Assume that $\cL_{\sZ/\sX} \in \Dperf^{[-2,-1]}(\sZ)$, $\cL_{\sZ/\sX}^\fix \simeq 0$, and $Z$ has finite stabilizers.
      Then the natural transformation of endofunctors of $\D^T(X)\inv$
      \[
        i_*^\Sigma i^*_\Sigma
        \xrightarrow{\eul^{-1}} i_*^\Sigma i^*_\Sigma \vb{-c}
        \xrightarrow{\tr^\Sigma_i} \id,
      \]
      where $\eul^{-1}$ is an inverse of $\eul_{\cN_{\sZ/\sX}}^\Sigma$, is an inverse to the unit $\id \to i_*^\Sigma i^*_\Sigma$.
    \end{thm}

  \ssec{Example: $\on{Hilb}^n(\A^3)$}

    We return to the example of the $T$-equivariant derived critical locus $X=\on{Crit}(f)$ with $X_\cl\simeq \on{Hilb}^n(\A^3)$ (\examref{exam:hilb}).
    Recall that $X$ is quasi-smooth, although $\on{Hilb}^n(\A^3)$ is singular for $n\ge 4$.

    As we saw in \examref{exam:hilb}, every object $\cF \in \D^T(X)\inv$ decomposes canonically via natural isomorphisms (we omit the decorations ``$\Sigma$'')
    \[ \mrm{unit} : \cF \to \bigoplus_{x \in X^T} i_{x,*}i_x^*(\cF), \]
    indexed by the plane partitions of $n$.
    Here $i_x : \{x\} \hook X$ is the inclusion of a point $x \in X^T$.
    \thmref{thm:catloc} implies that the projector of $\cF$ onto its $x$-summand is $(-1)^n \cdot \tr_{i,x} \circ \mrm{unit}_x$, where $\mrm{unit}_x : \cF \to i_{x,*}i_x^*\cF$ and $\tr_{i,x} : i_{x,*}i_x^*\cF \to \cF$ are the $x$-components of $\mrm{unit}$ and $\tr_i$, for $i : X^T \hook X$.

    Note that $(-1)^n$ is precisely the value of the Behrend function of $X$ at $x$.
    The sum of these values over $x \in X^T$, the weighted Euler characteristic of $X$, is thus $(-1)^n \cdot p_n$ where $p_n$ is the number of plane partitions of $n$ (see \cite[Prop.~4.2]{BehrendFantechi}).
    This can thus be regarded as a categorification of the computation of the degree-zero DT invariant of $X$.

    To see why \thmref{thm:catloc} applies here and how it recovers the claim above, let $V_x$ be the Zariski tangent space of $X_\cl$ at $x$.
    Write $Z := X^T$ and $\sZ = [Z/T] \simeq X^T \times BT$.
    Since $\cL_Z \simeq 0$, we have $\cN_{\sZ/\sX}|_x \simeq \cL_X|_x$, and the $(-1)$-shifted symplectic form identifies $\cL_X|_x$ with $V_x[1] \oplus V_x^\vee$; thus $\cL_{\sZ/\sX} \in \Dperf^{[-2,-1]}(\sZ)$, and $\cN_{\sZ/\sX}$ has virtual rank $0$.
    Moreover, this identification is $T$-equivariant, since $f$ is $T$-invariant and hence the $(-1)$-shifted symplectic form has weight $0$.
    Since $T$ acts on $V_x$ with nonzero weights \cite[Lem.~4.1(b)]{BehrendFantechi}, we have $\cN_{\sZ/\sX}^\fix \simeq 0$, so the hypotheses of \thmref{thm:catloc} hold.
    It remains only to observe that $e(\cN_{\sZ/\sX}|_x) = e(V_x^\vee) \cdot e(V_x)^{-1} = (-1)^{\dim(V_x)} = (-1)^n$ ($\dim V_x \equiv n \pmod 2$ by \cite[Lem.~4.1(c)]{BehrendFantechi}), so that the Euler transformation $\eul_{\cN_{\sZ/\sX}} : \cG \to \cG$ on $\D^T(Z)\inv \simeq \prod_x \D(BT)\inv$ is $((-1)^n \cdot \id_{\cG_x})_x$.

  \ssec{Categorified self-intersection formulas}

    The proof of \thmref{thm:catloc} will require the following categorifications of Proposition~2.4(i) and Corollary~2.5 in \cite{Virloc}.
    A closed immersion $i : Z \to X$ is \emph{homotopically smooth} if its conormal complex $\cN_{Z/X} = \cL_{Z/X}[-1]$ is perfect; its derived normal bundle is $N_{Z/X} := \V_Z(\cN_{Z/X}) \to Z$.

    \begin{prop}\label{prop:sp-counit}
      Suppose given a commutative square
      \[
        \begin{tikzcd}
          Z' \ar{r}{i'}\ar{d}{p}
          & X' \ar{d}{q}
          \\
          Z \ar{r}{i}
          & X
        \end{tikzcd}
      \]
      in $\Stk$ where $i$ and $i'$ are homotopically smooth closed immersions, $q$ is proper, and the square is excessive (i.e., cartesian on classical truncations and such that the induced morphism $N_{Z'/X'} \to N_{Z/X} \times_Z Z'$ is a closed immersion; see \cite[Def.~3.30]{HKR}).
      Denote by $\pi': N_{Z'/X'} \to Z'$, $\pi: N_{Z/X} \to Z$ the derived normal bundles to $i'$ and $i$, respectively, and by $N_\Delta : N_{Z'/X'} \to N_{Z/X}$ the induced (proper) morphism.
      Then the diagram
      \[\begin{tikzcd}
        i_!\pi_!\pi^*i^* \ar{rr}{\sp_i}\ar[swap]{d}{\mrm{unit}}
        & & \id \ar{d}{\mrm{unit}}
        \\
        i_!\pi_! N_{\Delta,!}N_\Delta^* \pi^*i^* \ar[equals]{r}
        & q_!i'_!\pi'_!\pi'^*i'^*q^* \ar{r}{\sp_{i'}}
        & q_!q^*,
      \end{tikzcd}\]
      where the lower isomorphism is by functoriality, commutes up to canonical homotopy.
    \end{prop}

    \begin{proof}
      By \cite[Prop.~3.37]{HKR}, the induced morphism $D_\Delta : D_{Z'/X'} \to D_{Z/X}$ on the derived deformations to the normal bundle is proper.
      The commutative diagram of complementary closed/open immersions
      \begin{equation*}
        \begin{tikzcd}
          N_{Z'/X'} \ar{r}{u'}\ar{d}{N_\Delta}
          & D_{Z'/X'} \ar[leftarrow]{r}{v'}\ar{d}{D_\Delta}
          & X' \times \bG_m\ar{d}{q\times\id}
          \\
          N_{Z/X} \ar{r}{u}
          & D_{Z/X} \ar[leftarrow]{r}{v}
          & X \times \bG_m
        \end{tikzcd}
      \end{equation*}
      gives rise to the commutative diagram
      \[\begin{tikzcd}
        v_!v^* \ar{r}\ar{d}
        & \id \ar{r}\ar{d}
        & u_!u^*\ar{d}
        \\
        v_!(q\times\id)_!(q\times\id)^*v^*\ar{r}\ar[equals]{d}
        & D_{\Delta,!}D_\Delta^* \ar{r}\ar[equals]{d}
        & u_!N_{\Delta,!}N_\Delta^*u^* \ar[equals]{d}
        \\
        D_{\Delta,!}v'_!v'^*D_\Delta^* \ar{r}
        & D_{\Delta,!}D_\Delta^* \ar{r}
        & D_{\Delta,!}u'_!u'^*D_\Delta^*
      \end{tikzcd}\]
      where the rows are localization triangles and the upper vertical arrows are unit transformations.
      Now the claim follows by unravelling the definition of the specialization transformation.
    \end{proof}

    \begin{cor}\label{cor:sp-counit:Self-intersection}
      Let $i: Z \to X$ be a homotopically smooth closed immersion in $\Stk$.
      Then the diagram
      \[\begin{tikzcd}
        i_!\pi_!\pi^*i^* \ar{rr}{\mrm{unit}}\ar[equals]{d}
        & & i_!\pi_!0_!0^*\pi^*i^* \ar[equals]{d}
        \\
        i_!\pi_!\pi^*i^* \ar{r}{\sp_i}
        & \id \ar{r}{\mrm{unit}}
        & i_!i^*
      \end{tikzcd}\]
      commutes up to canonical homotopy, where $0 : Z \to N_{Z/X}$ denotes the zero section of the derived normal bundle $\pi : N_{Z/X} \to Z$.
    \end{cor}
    \begin{proof}
      The self-intersection square
      \[\begin{tikzcd}
        Z \ar[equals]{r}\ar[equals]{d}
        & Z \ar{d}{i}
        \\
        Z \ar{r}{i}
        & X
      \end{tikzcd}\]
      satisfies the assumptions of \propref{prop:sp-counit} (see \cite[Ex.~3.35]{HKR}), and $\sp_{\id_Z} \simeq \id$.
    \end{proof}

    \begin{cor}[Categorical virtual self-intersection formula]\label{cor:antilacrosse}
      Let $X\in\Stk$ be quasi-compact with $T$-action, and let $Z \to X$ be a $T$-equivariant closed immersion such that the $T$-action on $Z$ is trivial.
      Denote by $i : \sZ \to \sX$ the induced morphism between $\sZ := [Z/T] \simeq Z \times BT$ and $\sX := [X/T]$.
      Assume that $\cL_{\sZ/\sX} \in \Dperf^{[-2,-1]}(\sZ)$, $\cL_{\sZ/\sX}^\fix \in \Dperf^{[-1,-1]}(\sZ)$, and $Z$ has finite stabilizers.
      Then the composite natural transformation
      \[
        i_*^\Sigma i^*_\Sigma \vb{-c}
        \xrightarrow{\tr_i^\Sigma} \id
        \xrightarrow{\mrm{unit}} i_*^\Sigma i^*_\Sigma
      \]
      of endofunctors of $\D^T(X)\inv$, is canonically identified with
      \[\eul_{\cN_{\sZ/\sX}}^\Sigma : i_*^\Sigma i^*_\Sigma \vb{-c} \to i_*^\Sigma i^*_\Sigma.\]
    \end{cor}

    \begin{proof}
      Let $\pi : N_{\sZ/\sX} \to \sZ$ denote the derived normal bundle and $0 : \sZ \to N_{\sZ/\sX}$ its zero section.
      Consider the following diagram of endofunctors of $\D^T(X)\inv$
      \[\begin{tikzcd}[column sep=40]
        i_* i^* \vb{-c} \ar{r}{\tr_\pi^{-1}}\ar[equal]{d}
        & i_*\pi_!\pi^*i^* \ar{rr}{\mrm{unit}}\ar[equal]{d}
        & & i_*\pi_!0_!0^*\pi^*i^*\ar[equal]{d}
        \\
        i_* i^* \vb{-c} \ar{r}{\tr_\pi^{-1}}
        & i_*\pi_!\pi^*i^* \ar{r}{\sp_i}
        & \id \ar{r}{\mrm{unit}}
        & i_* i^*,
      \end{tikzcd}\]
      where we omit the decorations ``$\Sigma$'' and $\tr_\pi^{-1}$ is the inverse of the isomorphism of \propref{prop:Htpy-Inv}.
      The right-hand square commutes by \corref{cor:sp-counit:Self-intersection}.
      The upper row is $\eul_{\cN_{\sZ/\sX}}^\Sigma$ by \constrref{constr:postcarotid}, and the lower row is $\mrm{unit} \circ \tr_i^\Sigma$ by \constrref{constr:undeliberating}.
    \end{proof}

  \ssec{Proof of \thmref{thm:catloc}}

    Repeat the proof of \thmref{thm:catlocsm}, substituting in \corref{cor:antilacrosse} for \propref{prop:ethnogeographically}.

\section{Cohomological localization and integration formulas}

  We now derive from \thmref{thm:catloc} the localization formula in cohomology and Borel--Moore homology.
  The latter recovers \cite[Thm.~3.3, Cor.~3.4]{Virloc}.

  We begin by defining Euler classes and Gysin maps.
  Recall that for every $X \in \Stk$ with $T$-action we have a canonical functor \eqref{eq:bedull}
  \begin{equation}\label{eq:daggletail}
    F_\Sigma : \D^T(X)\inv \to \Mod_{\Ccoh_T(X)\per\inv}
  \end{equation}
  with the property that $F_\Sigma(L_\Sigma\cF) \simeq \Ccoh_T(X;\cF)\per\inv$ for any $\cF \in \D^T(X)$.
  Dually, for any fixed $\cG \in \D^T(X)$ there is a canonical functor
  \begin{equation}\label{eq:noseband}
    G_\Sigma^\cG : \D^T(X)\inv \to (\Mod_{\Ccoh_T(X)\per\inv})^\op
  \end{equation}
  with $G_\Sigma^\cG(L_\Sigma \cF) \simeq \bigoplus_{n\in\Z} \Maps_{\D^T(X)}(\cF, \cG\vb{n})\inv$.

  \begin{constr}[Euler class]
    Let $X \in \Stk$ be quasi-compact with finite stabilizers.
    Suppose given a perfect complex
    \[
      \cE \in \Dperf^{[-1,0]}(\sX)
      \quad \text{with}~\cE^\fix \in \Dperf^{[0,0]}(\sX)
    \]
    on $\sX := X \times BT$.
    The Euler transformation $\eul_\cE^\Sigma$ (\constrref{constr:postcarotid}) gives rise to a morphism
    \[
      \eul_\cE : L_\Sigma \un_\sX \to L_\Sigma \un_\sX\vb{r},
    \]
    where $r$ is the virtual rank of $\cE$.
    Applying $F_\Sigma$ yields a canonical morphism of the form
    \[
      \Ccoh_T(X)\per\inv
      \to \Ccoh_T(X)\per\inv
    \]
    which by $\Ccoh_T(X)\per\inv$-linearity amounts to an element
    \begin{equation}
      e(\cE) \in \Ccoh_T(X)\vb{r}\inv,\
    \end{equation}
    we call the \emph{Euler class} (or \emph{top Chern class}).
    If $\cE^\fix \simeq 0$, then $\eul_\cE^\Sigma$ is invertible (\propref{prop:trophoderm}), hence $e(\cE) \in \Ccoh_T(X)\per\inv$ is invertible.
    By cap product, $e(\cE)$ gives rise for every $\Ccoh_T(X)\per$-module spectrum $M$ to a map
    \[
      e(\cE) : M\inv \to M\inv,
    \]
    invertible when $\cE^\fix \simeq 0$.
  \end{constr}

  \begin{constr}[Gysin maps]
    Let $X\in\Stk$ be quasi-compact with $T$-action, and let $Z \to X$ be a $T$-equivariant closed immersion where the $T$-action on $Z$ is trivial.
    Write $\sZ := [Z/T] \simeq Z \times BT$ and $\sX := [X/T]$.
    Assume that $\cL_{\sZ/\sX} \in \Dperf(\sZ)^{[-2,-1]}$, $\cL_{\sZ/\sX}^\fix \in \Dperf(\sZ)^{[-1,-1]}$, and $Z$ has finite stabilizers.
    Let $i : \sZ \to \sX$ denote the inclusion and let $c$ be the virtual rank of $\cL_{\sZ/\sX}[-1]$.
    \begin{defnlist}
      \item
      For any object $\cF \in \D^T(X)$, consider the morphism
      \[
        L_\Sigma i_! i^* \cF\vb{-c}
        \simeq i_!^\Sigma i^*_\Sigma L_\Sigma \cF\vb{-c}
        \xrightarrow{\tr_i^\Sigma} L_\Sigma\cF
      \]
      given by the trace transformation \eqref{eq:dietotherapy}.
      Applying $F_\Sigma$ \eqref{eq:daggletail} yields a canonical map of $\Ccoh_T(X)\per\inv$-module spectra
      \begin{equation}\label{eq:strengthful}
        \begin{multlined}
          i_! : \Ccoh_T(Z; i^*\cF)\per\inv
          \simeq \Ccoh_T(X; i_!i^*\cF)\per\inv\\
          \to \Ccoh_T(X; \cF)\per\inv.
        \end{multlined}
      \end{equation}
      We call \eqref{eq:strengthful} the \emph{Gysin push-forward} along $i$.

      \item
      The trace \eqref{eq:dietotherapy} gives rise to a morphism
      $$L_\Sigma i_!\un_\sZ\vb{-c} \to L_\Sigma \un_\sX.$$
      For any $\cF \in \D^T(X)$ we may apply the functor $G_\Sigma^{\cF}$ \eqref{eq:noseband} to get a canonical map of $\Ccoh_T(X)\per\inv$-modules
      \[
        \bigoplus_{n\in\Z} \Maps_{\D^T(X)}(\un_\sX, \cF\vb{n})\inv
        \to \bigoplus_{n\in\Z} \Maps_{\D^T(X)}(i_!\un_\sZ\vb{-c}, \cF\vb{n})\inv.
      \]
      By adjunction, this amounts to a canonical $\Ccoh_T(X)\per\inv$-module map
      \begin{equation}
        i^! : \Ccoh_T(X; \cF)\per\inv
        \to \Ccoh_T(Z; i^!\cF)\per\inv.
      \end{equation}
      We think of this as a Gysin pull-back in Borel--Moore homology (compare \remref{rem:keelhale}).
    \end{defnlist}
  \end{constr}

  \begin{cor}[Co/homological localization formula]\label{cor:truck}
    Let $X\in\Stk$ be quasi-compact with $T$-action, and let $Z \to X$ be a $T$-equivariant closed immersion such that $\St^T_X(x) \subsetneq T_x$ for every geometric point $x \in X \setminus Z$, and such that the $T$-action on $Z$ is trivial.
    Denote by $i : \sZ \to \sX$ the induced morphism between $\sZ := [Z/T] \simeq Z \times BT$ and $\sX := [X/T]$.
    Assume that $\cL_{\sZ/\sX} \in \Dperf^{[-2,-1]}(\sZ)$, $\cL_{\sZ/\sX}^\fix \simeq 0$, and $Z$ has finite stabilizers.
    Then for every $\cF \in \D^T(X)$ we have:
    \begin{thmlist}
      \item
      There is a canonical identification
      \[
        i_!\left(- \cap e(\cN_{\sZ/\sX})^{-1}\right)
        \simeq (i^*)^{-1}
      \]
      of maps $\Ccoh_T(Z; \cF)\per\inv \to \Ccoh_T(X; \cF)\per\inv$.

      \item
      There is a canonical identification
      \[
        i^!(-) \cap e(\cN_{\sZ/\sX})^{-1}
        \simeq (i_*)^{-1}
      \]
      of maps $\Ccoh_T(X; \cF)\per\inv \to \Ccoh_T(Z; i^!\cF)\per\inv$.
    \end{thmlist}
  \end{cor}

  \begin{proof}
    Follows immediately from \thmref{thm:catloc} by unravelling definitions.
  \end{proof}

  \begin{cor}[Integration formula]\label{cor:affirm}
    Let the notation be as in \corref{cor:truck}.
    Suppose that $\sX$ is proper representable over some $S \in \Stk$.
    Denote by $f : \sX \to S$ and $g : \sZ \to S$ the structural morphisms.
    Then for every $\cF \in \D(S)$ we have:
    \begin{thmlist}
      \item
      If $f : \sX \to S$ is quasi-smooth, then there is a canonical identification\footnote{%
        Here $i_!$ is the Gysin map of \eqref{eq:strengthful}.
        When $g$ is quasi-smooth, we can simplify $f_!i_! \simeq g_!$.
      }
      \[
        f_!(-) \simeq f_!i_!(i^*(-) \cap e(\cN_{\sZ/\sX})^{-1})
      \]
      of maps $\Ccoh_T(X; f^*\cF)\per\inv \to \Ccoh(S; \cF)\per\inv$.

      \item
      There is a canonical identification
      \[
        f_*(-) \simeq g_*(i^!(-) \cap e(\cN_{\sZ/\sX})^{-1})
      \]
      of maps
      \[
        \Ccoh_T(X; f^!\cF)\per\inv \to \Ccoh(S; \cF)\per\inv.
      \]
    \end{thmlist}
  \end{cor}
  \begin{proof}
    Since $f$ is proper representable, there is a canonical isomorphism $f_! \simeq f_*$.
    In particular, one has a cohomological Gysin push-forward $f_!$ (where $f$ is quasi-smooth) and Borel--Moore homological push-forwards $f_*$ and $g_*$, respectively.
  \end{proof}

  \begin{rem}
    For the Betti and étale weaves, or motivic sheaves with rational coefficients, \corref{cor:affirm} holds more generally without assuming that $\sX \to S$ is representable.
    This is because in that case we have the canonical isomorphism of functors $f_! \simeq f_*$ for any proper morphism $f$ representable by Deligne--Mumford stacks (see \cite[Thm.~D]{propdesc}).
  \end{rem}

Institute of Mathematics, Academia Sinica, 10617 Taipei, Taiwan

School of Mathematics, Tata Institute of Fundamental Research, Mumbai 400005, India

\end{document}